\documentclass{amsart}
\usepackage[utf8]{inputenc}
\usepackage[asymmetric,vmargin=3cm,hmargin=3cm,marginpar=2cm]{geometry}
\usepackage{float}

\usepackage{amsmath,amssymb,amsfonts,amsthm,amscd,graphicx,psfrag,epsfig,fge}
\usepackage{bbm}
\usepackage{shuffle}
\usepackage{tikz-cd}

\usepackage{enumitem}
\setlist[enumerate]{itemsep=1ex,leftmargin=2em, label={\rm(\roman*)}}
\setlist[itemize]{itemsep=1ex,leftmargin=2em}

\usepackage{xcolor}
\usepackage{hyperref}
\usepackage{mathtools}

\def\equationautorefname~#1\null{Equation~(#1)\null}

\makeatletter
\def\sectionautorefname{\S\@gobble}
\def\subsectionautorefname{\S\@gobble}
\makeatother

\DeclareMathOperator{\Li}{Li}

\newcommand{\Q}{\mathbb{Q}}
\newcommand{\Z}{\mathbb{Z}}
\newcommand{\R}{\mathbb{R}}

\renewcommand{\L}{\mathcal{L}}
\DeclareMathOperator{\ch}{char}

\newcommand{\lra}{\longrightarrow}

\renewcommand{\Im}{\operatorname{im}}

\DeclareMathOperator{\Conf}{Conf}
\DeclareMathOperator{\Ext}{Ext}

\newcommand{\QHod}{\text{$\Q$-$\mathrm{Hod}$}}
\newcommand{\RHod}{\text{$\R$-$\mathrm{Hod}$}}

\usepackage{thmtools}

\declaretheorem[
style=plain,
name=Theorem,
numbered=yes,
refname={Theorem,Theorems},
Refname={Theorem,Theorems}
]{theorem}
\declaretheorem[
style=plain,
name=Proposition,
numberlike=theorem,
refname={Proposition,Propositions},
Refname={Proposition,Propositions}
]{proposition}
\declaretheorem[
style=definition,
name=Remark,
numberlike=theorem,
refname={Remark,Remarks},
Refname={Remark,Remarks}
]{remark}
\declaretheorem[
style=plain,
name=Conjecture,
numberlike=theorem,
refname={Conjecture,Conjectures},
Refname={Conjecture,Conjectures}
]{conjecture}
\declaretheorem[
style=definition,
name=Definition,
numberlike=theorem,
refname={Definition,Definitions},
Refname={Definition,Definitions}
]{definition}

\declaretheorem[
style=plain,
name=Corollary,
numberlike=theorem,
refname={Corollary,Corollaries},
Refname={Corollary,Corollaries}
]{corollary}
\declaretheorem[
style=plain,
name=Lemma,
numberlike=theorem,
refname={Lemma,Lemmas},
Refname={Lemma,Lemmas }
]{lemma}
\declaretheorem[
style=plain,
name=Question,
numberlike=theorem,
refname={Question,Questions},
Refname={Question,Questions}
]{question}

\newcommand{\LiL}{\Li^{\L}}
\newcommand{\IL}{\operatorname{I}^{\L}}

\let\phi\varphi

\DeclareMathOperator{\M}{\mathcal{M}}
\DeclareMathOperator{\I}{I}
\DeclareMathOperator{\id}{id}
\DeclareMathOperator{\Id}{Id}
\DeclareMathOperator{\gr}{gr}
\newcommand{\CoLie}{\operatorname{Lie}^c}

\allowdisplaybreaks

\def\semi{\,{\boldsymbol{;}}\,}
\def\={\mathrel{\,{=}\,}}

\def\+{\!+\!}
\def\-{\!-\!}

\DeclareMathOperator{\cor}{Cor}

\newcommand{\dd}{\mathrm{d}}

\DeclareMathOperator{\Sp}{Sp}

\newcommand{\refAall}{\ref{Acycle}--\ref{Aone}}
\renewcommand{\refAall}{\ref{Acycle}--\ref{Aone}}

\newcommand{\llp}{(\mkern-4mu(}
\newcommand{\rrp}{)\mkern-4mu)}

\usepackage{todonotes}

\title{The Hopf algebra of formal multiple polylogarithms}

\author[Charlton]{Steven Charlton}
\address[Charlton]{Max Planck Institute for Mathematics, Vivatsgasse 7, 53111 Bonn, Germany}
\email{charlton@mpim-bonn.mpg.de}

\author[Matveiakin]{Andrei Matveiakin}
\address[Matveiakin]{}
\email{a.matveiakin@gmail.com}

\author[Radchenko]{Danylo Radchenko}
\address[Radchenko]{Laboratoire Paul Painlev\'e, Universit\'e de Lille, F-59655 Villeneuve d'Ascq,	France}
\email{danradchenko@gmail.com}

\author[Rudenko]{Daniil Rudenko}
\address[Rudenko]{Department of Mathematics, University of Chicago, 5801 S Ellis Ave, 60637 Chicago, IL, USA}
\email{rudenkodaniil@gmail.com}

\begin{document}

\begin{abstract}
We define a Hopf algebra of polylogarithms of an arbitrary field, which is a candidate for a conjectural Hopf algebra of framed mixed Tate motives. Our definition is elementary and mimics Goncharov's construction of higher Bloch groups. We also discuss the Hodge and motivic realizations of the Hopf algebra of polylogarithms.
\end{abstract}

\maketitle

\tableofcontents

\section{Introduction}
\subsection{Multiple polylogarithms and Goncharov's program}
Multiple polylogarithms are multivalued functions of variables
$x_1,\dots,x_k\in \mathbb{C}$ 
depending on positive integer parameters $n_1,\dots,n_k\in \mathbb{N}$. In the polydisc   $|x_1|,|x_2|,\dots, |x_k| <1$ the multiple polylogarithm \( \Li_{n_1, \ldots, n_k} \) is  defined by the power series 
\begin{equation}\label{FormulaPolylogarithm}
\Li_{n_1,n_2,\dots, n_k}(x_1,x_2,\dots,x_k)=\sum_{0<m_1<m_2<\dots<m_k}\frac{x_1^{m_1} x_2^{m_2}\dots x_k^{m_k}}{m_1^{n_1}m_2^{n_2}\dots m_k^{n_k}}\,.
\end{equation}
The number \( n = n_1 + \cdots + n_k \) is called the \emph{weight} of the function $\Li_{n_1,\dots,n_k}$ and \( k \) is its \emph{depth}.

Multiple polylogarithms satisfy numerous functional equations, like the following {\it reflection} equation for dilogarithm discovered by Euler:
\begin{equation}\label{FormulaEulerEquation}
\Li_2(a)+\Li_2(1-a)=\frac{\pi^2}{6}-\log(a) \log(1-a) \,, \quad \text{for $0<a<1$.}
\end{equation}
Working with identities between multivalued functions can be rather tedious: even the simple equation $\log(a_1)+\log(a_2)=\log(a_1a_2)$  only holds modulo integer multiples of $2\pi i$. In general, it is known that any two branches of a multiple polylogarithm of weight $n$ differ by a product of $2\pi i$ with a $\Q$-linear combination of multiple polylogarithms of weight $n-1$, see \cite[\S2.5, \S2.6]{zhaoBook} for further details.  So, modulo multiples of $2\pi i$ (more precisely, modulo $\Q$-linear combinations of multiple polylogarithms of weight $n-1$ multiplied by $2\pi i $),~\eqref{FormulaEulerEquation} simplifies to 
\[
\Li_2(a)+\Li_2(1-a) \equiv -\log(a) \log(1-a) \pmod{2\pi i}
\]
and holds for all $a\in \mathbb{C}\setminus\{0,1\}$.

 Recall the connection between multiple polylogarithms, mixed Tate motives, and the algebraic $K$-theory of fields (for details see \cite{BD94}, \cite{Gon95B}, \cite{GonICM},\cite{Cle21}). Beilinson and Deligne conjectured that for every field $F$ there exists an abelian category of mixed Tate motives. Assuming this category exists, it would be equivalent to the category of graded comodules over a certain graded connected commutative Hopf algebra $\mathcal{H}^{\M}(F)$.  Then, for any  $a_1,\dots,a_k\in F$ and $n_1,\dots,n_k\in \mathbb{N}$, $n=n_1+\dots+n_k$ there would exist elements 
\[
\Li_{n_1,n_2,\dots, n_k}^{\M}(a_1,a_2,\dots,a_k)\in \mathcal{H}_n^{\M}(F)
\]
called {\it motivic multiple polylogarithms} whose properties reflect the properties of \eqref{FormulaPolylogarithm}. In particular, motivic multiple polylogarithms would satisfy certain relations which hold for polylogarithms as analytic functions modulo multiples of $2\pi i$. Goncharov~\cite[Conjectures 1.9 and 7.4]{Gon01} conjectured that motivic multiple polylogarithms span $\mathcal{H}_n^{\M}(F)$ as a $\Q$-vector space and gave a candidate description of all relations in weight three. For $n\geq 1$, denote by $H^{i}(\mathcal{H}^{\M}(F),\Q)_n$  the $i$-th cohomology group of the $n$-th graded component of the cobar complex of the Hopf algebra $\mathcal{H}^{\M}(F)$, given by
\[
0 \lra \mathcal{H}^{\M}_{n}(F)\lra \bigoplus_{\substack{n_1+n_2=n \\ n_1,n_2\geq 1}}\mathcal{H}^{\M}_{n_1}(F)\otimes \mathcal{H}^{\M}_{n_2}(F) \lra \cdots 
\,, \]
with \( \mathcal{H}^{\M}_n(F) \) placed in cohomological degree 1. The properties of mixed Tate motives would imply that
\begin{equation}\label{FormulaPolylogsKtheory}
H^{i}(\mathcal{H}^{\M}(F),\Q)_n \cong \gr_{\gamma}^{n} K_{2n-i}(F)_{\Q}\,,
\end{equation}
where $\gamma$ denotes the $\gamma$-filtration on the algebraic $K$-groups \cite{Sou85} and for any abelian group $A$, we denote by $A_\Q$ its rationalization $\Q\otimes_{\Z}A$.

The existence of the category of mixed Tate motives is known only for some classes of fields, but in particular it is known for number fields \cite{Lev93}, \cite{DG05}. Our goal is to give an elementary construction of a Hopf algebra $\mathcal{H}^\mathrm{f}(F)$ for an arbitrary field $F$ which we call a Hopf algebra of formal multiple polylogarithms. We expect it to be isomorphic to $\mathcal{H}^\mathcal{M}(F)$ when the latter exists. For number fields, we construct a {\it motivic realization}: a Hopf algebra morphism from $\mathcal{H}^\mathrm{f}(F)$ to $\mathcal{H}^\mathcal{M}(F)$. In the case when $F$ is a subfield of $\mathbb{C}$ we construct a {\it Hodge realization}: a Hopf algebra morphism from $\mathcal{H}^\mathrm{f}(F)$ to the Hopf algebra of framed mixed Hodge--Tate structures.

We will see that $\mathcal{H}^\mathrm{f}_{1}(F)$ is isomorphic to $F^{\times}_{\Q}$ and  $\mathcal{H}^\mathrm{f}_{2}(F)$ is closely related to the rationalization of the Bloch group $B_2(F)$ (see~\autoref{sec:wt2wt3} for a precise definition). We show how to prove identities in $\mathcal{H}^\mathrm{f}(F)$, and establish several well-known functional equations for multiple polylogarithms do hold in this setting. Thanks to realization homomorphisms, identities in $\mathcal{H}^\mathrm{f}(F)$ imply corresponding identities between framed mixed Hodge--Tate structures and framed mixed Tate motives. 

Our definition of $\mathcal{H}^\mathrm{f}(F)$ mimics Goncharov's construction of higher Bloch groups in \cite{Gon95B}, (see \autoref{sec:gonconj} for a comparison; a construction related to Goncharov's was also given by Zagier in~\cite[\S8]{zagier}). Two constructions similar to ours already appear in the literature: \cite[\S3]{GoncharovMSRI} and \cite{GKLZ}. The main difference is in the allowed functional equations and in the treatment of specializations. In particular, we only allow ``rational'' functional equations, but we define specializations without any general position assumptions. This allows us to construct Hodge and motivic realizations, which are crucial for applications.

\subsection{Outline of the paper}
The main goal of the paper is to define the Hopf algebra $\mathcal{H}^\mathrm{f}(F)$ for an arbitrary field $F$. By Milnor-Moore, a graded connected commutative Hopf algebra can be reconstructed from its Lie coalgebra of indecomposables as the universal coenveloping coalgebra. So, it is sufficient to define a Lie coalgebra $\mathcal{L}^\mathrm{f}(F)$, which we do in \autoref{SectionAn}--\ref{SectionRelations}; the main steps of the construction are as follows.

For motivation, we return to the conjectural motivic setting. There is another family of elements in $\mathcal{H}^{\M}(F)$ called {\it motivic iterated integrals}.  Motivic iterated integrals were introduced in \cite{Gon01} and are motivic counterparts of the following integrals:
    \[ \I(x_0; x_1, \ldots, x_n; x_{n+1}) =
    \int_{x_0 < t_1 < \cdots < t_N < x_{N+1}}
    \frac{\dd t_1}{t_1-x_1}
    \wedge \frac{\dd t_2}{t_2-x_2}
    \wedge \cdots \wedge
    \frac{\dd t_n}{t_n-x_n}.
    \]
Motivic multiple polylogarithms and motivic iterated integrals are related by the {\it Leibniz-Kontsevich formula} (\cite[Thm. 2.2]{Gon01}, \cite[p. 2]{Goncharov_Manin_2004}),
\begin{equation} \label{EquationLeibnizKontsevich}
	\Li^{\M}_{n_1,n_{2},\ldots,n_k}(a_1,a_2,\dots,a_k)
 = (-1)^{k}\I^{\M}(0;1,\underbrace{0,\dots,0,a_1}_{n_1},\underbrace{0,\dots,0,a_1a_2}_{n_2},\dots,\underbrace{0,\dots,0;a_1a_2\dots a_{k}}_{n_k})\,.
\end{equation}
In particular, they span the same subspace of $\mathcal{H}^{\M}(F)$.
Next, let $\mathcal{L}^{\M}(F)$ be the Lie coalgebra of indecomposables of $\mathcal{H}^{\M}(F)$. In \cite{Gon19} Goncharov introduced {\it motivic correlators} 
\[
\cor^{\M}(x_0,\dots, x_n) \in \mathcal{L}^{\M}(F)
\]
which are framed mixed Tate motives associated to the pro-unipotent completion of the fundamental group of the punctured projective line.
These elements are related to motivic iterated integrals:
\begin{equation} \label{FormulaIntegralsCorrelators}
\IL(x_0; x_1, \ldots, x_n; x_{n+1}) = 
\cor^{\M}(x_1,\ldots, x_n, x_{n+1}) - \cor^{\M}(x_0, x_1,\ldots, x_n).
\end{equation}
All three families of elements have their own advantages and role in the Goncharov program. We will eventually introduce all three families of elements in  $\mathcal{L}^\mathrm{f}(F)$ but it is convenient to use correlators in the definition and define iterated integrals and multiple polylogarithms in terms of them.

We now outline the construction of $\mathcal{L}^\mathrm{f}(F)$. For a finite field $F$ we declare $\mathcal{L}^\mathrm{f}_n(F)=0$, as is consistent with the motivic picture. Henceforth we assume that the field $F$ is infinite. Our first step is the definition of a Lie coalgebra $\mathcal{A}(F)$ in \autoref{SectionAn} generated as a rational vector space by symbols $\llp x_0,x_1,\dots,x_n \rrp$ for $x_i\in F$ subject to some simple relations which are known to hold for motivic correlators. In \autoref{SectionSpecializationAn} we discuss the specialization map on $\mathcal{A}(F)$, which is the algebraic counterpart of a limit for continuous functions. The Lie coalgebra $\mathcal{L}^\mathrm{f}(F)$ will be a quotient $\mathcal{A}(F)/\mathcal{R}(F)$ by a certain graded Lie coideal, and symbols $\llp x_0,x_1,\dots,  x_n \rrp$ project to elements $\cor(x_0,\dots, x_n)\in \mathcal{L}^\mathrm{f}_n(F)$ called correlators. 
 
The next step  in \autoref{SectionRelations} is to define the space of relations $\mathcal{R}_{n}(F)\subseteq \mathcal{A}_{n}(F)$, which we do inductively. Consider an element $R\in \mathcal{A}_{n}(F(t))$ such that the projection of $\delta(R)$ to $\bigwedge^2 \mathcal{L}^\mathrm{f}(F(t))$ vanishes, where $\mathcal{L}^\mathrm{f}_{k}(F')$ for $k<n$ and all fields \( F' \), is already defined by the induction hypothesis. An element $R$ defines a class in $H^{1}(\mathcal{L}^\mathrm{f}(F(t)),\Q)_n$ which we conjecture to be isomorphic to $H^{1}(\mathcal{L}^{\M}(F(t)),\Q)_n\cong \gr_{\gamma}^{n} K_{2n-1}(F(t))_\Q$. Here $H^k(\mathcal{L},\Q)$ denotes the $k$-th Chevalley-Eilenberg cohomology group of the Lie coalgebra $\mathcal{L}$. The Beilinson-Soul{\'e} vanishing conjecture implies that for $n\geq 2$ the map
\[
\gr_{\gamma}^{n} K_{2n-1}(F)_\Q \lra \gr_{\gamma}^{n} K_{2n-1}(F(t))_\Q \,,
\]
induced by the inclusion $F\hookrightarrow F(t)$, is an isomorphism. Thus $R$ should be constant. Guided by this logic, we define the space $\mathcal{R}_{n}(F)$ as a span of the elements $\Sp_{t\to 0}(R)-\Sp_{t\to 1}(R)$ for $R$ as above, where  $\Sp_{t\to t_0}$ is the specialization map defined in \autoref{SectionSpecializationAn}.

 In \autoref{SectionCorrelatorProperties} we prove some of the well-known properties of correlators and in \autoref{sec:IteratedIntegrals}--\ref{sec:propertiesII} we introduce formal analogues of motivic iterated integrals and prove their basic properties. As was explained above, we define $\mathcal{H}^\mathrm{f}(F)$ as the universal coenveloping coalgebra of $\mathcal{L}^\mathrm{f}(F).$ In \autoref{SectionHopfAlebra} we lift the iterated integrals to the Hopf algebra. That allows to define multiple polylogarithms in \autoref{SectionMultiplePolylogarithms} and introduce the depth filtration on $\mathcal{H}^\mathrm{f}(F)$.
In the remaining part of \autoref{sec:MultiplePolylogarithms} we prove in our setting some of the known properties of multiple polylogarithms.

The Hodge and motivic realizations are discussed in  \autoref{SectionRealizations}.  For $F$ a subfield of $\mathbb{C},$ the Hopf algebra $\mathcal{H}^\mathrm{f}(F)$ can be mapped to the Hopf algebra of framed mixed Hodge--Tate structures $\mathcal{H}^{\QHod}(F)$. Thus every relation which holds in 
$\mathcal{H}^\mathrm{f}(F)$ leads to the corresponding relation between framed mixed Hodge--Tate structures. Goncharov~\cite[\S1.11]{Gon19} constructed a $\Q$-linear map $p_{\R}\colon \mathcal{H}^{\QHod}(F)\lra \mathbb{R},$ called the \emph{canonical real period map}, which allows us to associate certain real numbers to the elements of $\mathcal{H}^\mathrm{f}(F)$. For instance, the canonical real period of $\log(a)\in \mathcal{H}^\mathrm{f}_{1}(F)$ equals $\log |a|\in \mathbb{R}$ and 
of $\Li_2(a)\in \mathcal{H}^\mathrm{f}_{2}(F) \) equals \( \operatorname{Im}(\Li_2(a) +  \log |a| \log(1-a) )$, the Bloch-Wigner dilogarithm of~$a.$
 In \autoref{SectionMotivicRealization} we define a {\it motivic realization} 
\[
r_{\M}\colon \mathcal{H}^\mathrm{f}_n(F)\longrightarrow \mathcal{H}_n^{\M}(F).
\] 
We conjecture that the map $r_{\M}$ is an isomorphism. 

\subsection*{Acknowledgements}
We would like to thank Cl\'ement Dupont, Herbert Gangl, Alexander Kupers,  Marc Levine, Ismael Sierra, and the anonymous referees for numerous comments, remarks and suggestions which helped to improve the exposition.

\section{Construction of the Lie coalgebra of formal multiple polylogarithms}\label{SectionLieCoalgebraMP} 
In this and the following section we define the  Hopf algebra of formal multiple polylogarithms $\mathcal{H}^\mathrm{f}(F)$. We start with constructing its Lie coalgebra of indecomposables $\mathcal{L}^\mathrm{f}(F)$. The  Hopf algebra $\mathcal{H}^\mathrm{f}(F)$ is defined as the universal coenveloping coalgebra $U^c\bigl(\mathcal{L}^\mathrm{f}(F)\bigr)$ of $\mathcal{L}^\mathrm{f}(F)$. We will refer to the grading $n$ in $\mathcal{L}^\mathrm{f}_n(F)$ and $\mathcal{H}^\mathrm{f}_n(F)$ as the \emph{weight}. For the case of a finite field $F=\mathbb{F}_q$ we put $\mathcal{L}^\mathrm{f}(\mathbb{F}_q)=0$ (see \autoref{PropositionFrobenius} for the rationale). In the rest of the paper we assume that $F$ is infinite.  

\subsection{Lie coalgebra \texorpdfstring{$\mathcal{A}_\bullet(F)$}{A\_bullet(F)}} \label{SectionAn}
Let $F$ be a field. Consider a $\Q$-vector space $\mathcal{A}_n(F)$ for $n\geq 1$ spanned by \( \llp x_0,x_1,\ldots,x_n \rrp \), for all ordered tuples of points $(x_0,x_1,\dots,x_n)\in F^{n+1}$, modulo relations
\begin{align}
    \tag*{$\bf (A_1)$}\label{Acycle} & \llp x_0,x_1,\dots,  x_n \rrp = \llp x_1,x_2,\dots,x_n,x_0 \rrp \,,\\
    \tag*{$\bf (A_2)$}\label{Atranslate} & \llp x_0+b,x_1+b,\dots,  x_n+b \rrp = \llp x_0,x_1,\dots,x_n \rrp \text{ for $b\in F$,}\\
    \tag*{$\bf (A_3)$}\label{Alog} & \llp 0,x \rrp + \llp 0,y \rrp = \llp 0,xy \rrp\ \text{for $ x,y \in F^{\times}$,}\\
    \tag*{$\bf (A_4)$}\label{Ascale} & \llp mx_0,mx_1,\dots, mx_n\rrp = \llp x_0,x_1,\dots,x_n\rrp \ \text{for $m\in F^{\times}$ and $n\geq 2$,}\\
    \tag*{$\bf (A_5)$}\label{Azero} & \llp 0,0,\dots,0 \rrp =0\,,\\
    \tag*{$\bf (A_6)$}\label{Aone} & \llp 1,0,\dots,0\rrp =0\,.
\end{align}
We put $\log(x):=\llp 0,x\rrp \in \mathcal{A}_1(F)$.

\begin{lemma}\label{LemmaWeightOne}
The map 
\[
u\colon \mathcal{A}_1(F) \lra F^{\times}_\Q
\]
sending $\llp x_0,x_1\rrp$ with $x_0\neq x_1$ to $(x_1-x_0)\in F^{\times}_\Q$ is an isomorphism.
\end{lemma}
\begin{proof}
First, we show that $u$ is well-defined. It is easy to see that $u$ annihilates \ref{Atranslate}--\ref{Aone}. To see that $u$ annihilates \ref{Acycle}, notice that $x=-x$ in $F^{\times}_{\Q}$ for every $x\in F^{\times}$. Thus
\[
u\bigl (\llp x_0,x_1\rrp \bigr)=x_1-x_0= (-1)(x_0-x_1)=x_0-x_1=u\bigl (\llp x_1,x_0\rrp \bigr)
\]
which shows that $u$ annihilates \ref{Acycle}.

To construct the map $v$ in the opposite direction, recall that $F^{\times}$ is a quotient of $\mathbb{Z}[F]$ by a subgroup generated by elements $[x]+[y]-[xy]$ for $x,y \in F^{\times}.$ Thus $F^{\times}_{\Q}$ is a quotient of  $\mathbb{Q}[F]$ by the subspace spanned by elements $[x]+[y]-[xy]$. We define the map $v\colon \mathbb{Q}[F] \lra \mathcal{A}_1(F)$ by the formula $v([x])=\log(x)$, which is  well-defined by \ref{Alog}. The maps $u$ and $v$ are mutually inverse, which finishes the proof.
\end{proof}

Recall the definition of a Lie coalgebra. For a vector space $V$ over $\Q$, consider maps $\tau \colon V \otimes V \lra V \otimes V$ sending $a\otimes b$ to $b\otimes a$ and $\eta \colon  V \otimes V \otimes V\lra V \otimes V\otimes V$ sending $a\otimes b \otimes c$ to $b\otimes c \otimes a$. A~Lie coalgebra is a $\Q$-vector space $\L$ together with a map $\delta \colon\L\lra \L\otimes \L$ such that $\tau \circ \delta = -\delta$ and  $(1+\eta +\eta^2)\circ (1\otimes \delta)\circ\delta=0$ (coJacobi identity). The subspace $\bigwedge^2 \L\subseteq \L\otimes \L$  is spanned by elements $v_1\wedge v_2 \coloneqq v_1\otimes v_2-v_2 \otimes v_1$; the image of the Lie cobracket is contained in this subspace.

We define a Lie cobracket 
$\delta\colon \mathcal{A}(F)\lra \bigwedge^2 \mathcal{A}(F)$
by the formula
\begin{equation*}
\delta \llp x_0,\dots, x_n\rrp =\sum_{j=0}^n\sum_{i=1}^{n-1}\llp x_{j}, x_{j+1}, \dots, x_{j+i}\rrp \wedge \llp x_{j},  x_{j+i+1}, \dots, x_{j+n}\rrp \,,
\end{equation*}
where we use the convention that $x_{i+n+1}=x_{i}$ for all $i$. This definition is motivated by the formula for the cobracket of motivic correlators~\cite[Eq.~(66)]{GR-zeta4}. \Autoref{fig:cobracket} shows a term in the cobracket for \( n = 5, j = 1, i = 2\).
\begin{figure}[H]
\begin{equation*}
	\begin{tikzpicture}[style = { thick}, baseline={([yshift=-.5ex]current bounding box.center)}]
	\draw (3,0) arc [start angle= 450, end angle = 90, radius=1.4]
	node[pos=0.0, name=x0, inner sep=0pt] {}
	node[pos=0.166, name=x1, inner sep=0pt] {}
	node[pos=0.333, name=x2, inner sep=0pt] {}
	node[pos=0.500, name=x3, inner sep=0pt] {}
	node[pos=0.583, name=xmid, inner sep=0pt] {}
	node[pos=0.666, name=x4, inner sep=0pt] {}
	node[pos=0.833, name=x5, inner sep=0pt] {}
;
	\filldraw
	(x0) circle [radius=0.1, fill=black] node[above, yshift=0.5ex] {$x_0$}
	(x1) circle [radius=0.1, fill=black] node[above right] {$x_1$}
	(x2) circle [radius=0.1, fill=black] node[below right] {$x_2$}
	(x3) circle [radius=0.1, fill=black] node[below, yshift=-0.5ex] {$x_3$}
	(x4) circle [radius=0.1, fill=black] node[below left] {$x_4$}
	(x5) circle [radius=0.1, fill=black] node[above left] {$x_5$}
;
	\draw [densely dotted, blue] (x1) edge (xmid);
	\end{tikzpicture}
\;\longrightarrow\;
\begin{tikzpicture}[style = { thick}, baseline={([yshift=-0.5ex]current bounding box.center)}]
\draw (3,0) arc [start angle=450, end angle = 90, radius=1]
node[pos=0.0, name=x1b, inner sep=0pt] {}
node[pos=0.333, name=x2b, inner sep=0pt] {}
node[pos=0.5, name=xphantom, inner sep=0pt] {}
node[pos=0.666, name=x3b, inner sep=0pt] {}
;
\filldraw
(x1b) circle [radius=0.1, fill=black] node[above, yshift=0.5ex] {$x_1$}
(x2b) circle [radius=0.1, fill=black] node[below right] {$x_2$}
(x3b) circle [radius=0.1, fill=black] node[below left] {$x_3$}
(xphantom) node[below, yshift=-0.5ex] {$\phantom{x_5}$}
;
\end{tikzpicture}
\; \wedge \;
\begin{tikzpicture}[style = { thick}, baseline={([yshift=-.5ex]current bounding box.center)}]
\draw (3,0) arc [start angle=450, end angle = 90, radius=1]
node[pos=0.0, name=x1b, inner sep=0pt] {}
node[pos=0.25, name=x4b, inner sep=0pt] {}
node[pos=0.5, name=x5b, inner sep=0pt] {}
node[pos=0.75, name=x0b, inner sep=0pt] {}
;
\filldraw
(x1b) circle [radius=0.1, fill=black] node[above, yshift=0.5ex] {$x_1$}
(x4b) circle [radius=0.1, fill=black] node[right, xshift=0.5ex] {$x_4$}
(x5b) circle [radius=0.1, fill=black] node[below, yshift=-0.5ex] {$x_5$}
(x0b) circle [radius=0.1, fill=black] node[left, xshift=-0.5ex] {$x_0$} ;
\end{tikzpicture}
\end{equation*}
\caption{A term in the cobracket for \( n = 5, j = 1, i = 2\).  We cut the circle from a labeled point to an arc, and make two new circles with labeled points. The cobracket is obtained by summing over all such cuts.}
\label{fig:cobracket}
\end{figure}
\begin{lemma} \label{LemmaCorrelatorsAffineInvariance}
The map $\delta$  is well-defined and satisfies the coJacobi identity.
\end{lemma}
\begin{proof} To see that $\delta$ is well-defined, we need to check that it respects relations \refAall. We check that for \ref{Ascale}; the rest is trivial.

First, notice that \ref{Acycle}, \ref{Atranslate} and \ref{Alog} imply that 
\begin{equation}\label{FormulaA1}
\llp mx,my\rrp =\log(m)+\llp x,y\rrp  \in \mathcal{A}_1(F)\,.
\end{equation}

We need to show the following identity in 
$\bigoplus_k \mathcal{A}_k(F)\wedge  \mathcal{A}_{n-k}(F)$:
\begin{equation} 
\label{CorrelatorAffineInvariance}
\begin{aligned}
&\sum_{i=1}^{n-1}\sum_{j=0}^n\llp mx_{j}, mx_{j+1}, \dots, mx_{j+i}\rrp \wedge \llp mx_{j}, mx_{j+i+1}, \dots, mx_{j+n}\rrp \\ 
&{} = \sum_{i=1}^{n-1}\sum_{j=0}^n\llp x_{j}, x_{j+1}, \dots, x_{j+i}\rrp \wedge \llp x_{j},  x_{j+i+1}, \dots, x_{j+n}\rrp \,.
\end{aligned}
\end{equation}
For $n=1$ there is nothing to check. For $n=2$, using~\eqref{FormulaA1} we have
\begin{align*}
&\llp mx_0,mx_1\rrp \wedge\llp mx_0,mx_2\rrp +\llp mx_1,mx_2\rrp \wedge\llp mx_1,mx_0\rrp +\llp mx_2,mx_0\rrp \wedge\llp mx_2,mx_1\rrp \\[1ex]
&{} \, = \, \begin{aligned}[t] & \llp x_0,x_1\rrp \wedge\llp x_0,x_2\rrp +\llp x_1,x_2\rrp \wedge\llp x_1,x_0\rrp +\llp x_2,x_0\rrp \wedge\llp x_2,x_1\rrp \\
& {} +\log(m) \wedge \bigl( \llp x_0,x_2\rrp -\llp x_0,x_1\rrp +\llp x_1,x_0\rrp -\llp x_1,x_2\rrp +\llp x_2,x_1\rrp -\llp x_2,x_0\rrp \bigr) 
\end{aligned} \\[1ex]
&{} \, = \,\llp x_0,x_1\rrp \wedge\llp x_0,x_2\rrp +\llp x_1,x_2\rrp \wedge\llp x_1,x_0\rrp +\llp x_2,x_0\rrp \wedge\llp x_2,x_1\rrp \,.
\end{align*}
Next, assume that $n>2$. Any term on the LHS of~\eqref{CorrelatorAffineInvariance} with $2\le i\le n-2$ (which lies in $\mathcal{A}_i(F)\wedge\mathcal{A}_{n-i}(F)$) coincides with the corresponding term on the RHS of~\eqref{CorrelatorAffineInvariance} by \ref{Ascale}. For the remaining terms (with $i\in\{1,n-1\}$) we calculate:
\begin{align*}
&\sum_{j=0}^n \llp mx_j,\dots,mx_{j+n-1}\rrp \wedge \bigl(\llp mx_j,mx_{j+n}\rrp -\llp mx_{j+n-1},mx_{j+n}\rrp \bigr)\\
&{} = \sum_{j=0}^n \llp x_j,\dots,x_{j+n-1}\rrp \wedge \bigl(\log(m)+\llp x_j,x_{j+n}\rrp-\log(m)-\llp x_{j+n-1},x_{j+n}\rrp \bigr) \\
&{} = \sum_{j=0}^n \llp x_j,\dots,x_{j+n-1}\rrp \wedge \bigl( \llp x_j,x_{j+n}\rrp-\llp x_{j+n-1},x_{j+n}\rrp\bigr) \,.
\end{align*}
Thus $\delta$ respects the \ref{Ascale} relation.

We leave the proof of the coJacobi identity as an exercise, it can be done by repeating the argument from~\cite[p.~436]{Gon01B}.
\end{proof}

\subsection{Specialization map on \texorpdfstring{$\mathcal{A}_n(F)$}{A\_n(F)}} \label{SectionSpecializationAn}

Consider a discrete valuation $\nu\colon F \lra \Z\cup\{\infty\}$, and the associated discrete valuation ring $\mathcal{O}=\{x\in F \mid \nu(x)\geq 0\}$, maximal ideal  $\mathfrak{m}=\{x\in F \mid \nu(x)\geq 1\}$, and residue field $\mathrm{k}=\mathcal{O}/\mathfrak{m}$. For $x\in\mathcal{O}$ we denote by $\overline{x}$ the image of $x$ the residue field $\mathrm{k}$. Recall that a uniformizer~$\pi$ is a generator of the ideal $\mathfrak{m}$. For $s\in \Z$ we call a collection of elements $x_0,\dots,x_n\in F$ {\it congruent modulo $\pi^s$} if $x_i-x_j\in \pi^s\mathcal{O}$ for $0\leq i,j\leq n$.

Given a  uniformizer $\pi\in \mathcal{O}$,  we define the specialization homomorphism 
\[
\Sp_{\nu, \pi}\colon \mathcal{A}_n(F)\lra \mathcal{A}_n(\mathrm{k}) 
\] 
in the following way. For  $x_0,\dots,x_n\in F$ not all equal to each other, consider the largest $s \in \Z$ such that all $x_i$ are congruent modulo $\pi^s$. Then there exist $a\in F$ and $y_0,\dots, y_n\in \mathcal{O}$ such that $x_i=y_i \pi^{s}+a$ and not all $y_i$ are congruent modulo $\pi$. In this case, we put 
\[
\Sp_{\nu, \pi}\llp x_0,\dots,x_n\rrp= \llp \overline{y}_0,\dots,\overline{y}_n\rrp\,.
\] 
By \ref{Atranslate}, the specialization does not depend on the choice of $a$. It is not hard to see that the map $\Sp_{\nu, \pi}$ is well-defined. By \ref{Ascale}, the specialization $\Sp_{\nu, \pi}$ does not depend on the choice of $\pi$ for $n\geq 2$. For simplicity we will usually suppress $\pi$ from the notation to write just \( \Sp_\nu \), even for $n=1$.

\begin{remark} \label{rem:defspec}
Our definition of specialization is motivated by the continuity property of Hodge correlators, see \cite[Theorem 11]{Mal20}.
\end{remark}

\begin{remark} \label{rem:spec0}
Observe that by \ref{Atranslate} $\mathcal{A}_n(F)$ is spanned by elements of the form $(0,x_1,\dots,x_n)$ with $x_i\in F$ and not all $x_i$ equal to $0$. Let $s$ be such that $\pi^s x_1,\dots,\pi^s x_n\in \mathcal{O}$ but not all $\pi^s x_i$ lie in $\mathfrak{m}$. Then 
\[
\Sp_{\nu}\llp 0,x_1,\dots,x_n\rrp = \llp 0,\overline{\pi^s x_1},\dots,\overline{\pi^s x_n}\rrp\in \mathcal{A}_n(\mathrm{k})\,.
\]
\end{remark}

\begin{lemma}\label{LemmaSpecializationCoproduct}
For any discrete valuation \( \nu \) on $F$, the Lie cobracket $\delta\colon \mathcal{A}(F)\lra \bigwedge^2 \mathcal{A}(F)$ commutes with the specialization, giving the following commutative diagram
\[
\begin{tikzcd}[row sep=large]
\mathcal{A}(F)  \arrow{r}{\delta} \arrow[d, "\Sp_\nu"] & \bigwedge^2 \mathcal{A}(F)  \arrow[d, "\Sp_\nu \wedge \Sp_\nu"] \\\mathcal{A}(\mathrm{k})  \arrow{r}{\delta}  & \bigwedge^2 \mathcal{A}(\mathrm{k})  \,.
\end{tikzcd}
\]
\end{lemma}
\begin{proof}
We need to show that 
\[
\delta \, \Sp_\nu\llp x_0,\dots, x_n\rrp=\sum_{j=0}^n\sum_{i=1}^{n-1}\Sp_\nu\llp x_{j}, x_{j+1}, \dots, x_{j+i}\rrp \wedge \Sp_\nu\llp x_{j},  x_{j+i+1}, \dots, x_{j+n}\rrp\,.
\]
We can assume that $n\geq 2$. Since $\Sp_\nu \llp x_{i_0},\dots,x_{i_\ell}\rrp = \Sp_\nu \llp \pi x_{i_0},\dots,\pi x_{i_\ell}\rrp$, we can assume that $x_0,\dots,x_n\in \mathcal{O}$. If $x_0,\dots,x_n$ are congruent modulo $\pi^s$ for some $s\geq 1$, we may assume that $s$ is the maximal number with this property and consider $y_0,\dots, y_n\in \mathcal{O}$ such that $x_i=y_i \pi^{s}+a$ and not all $y_i$ are congruent modulo $\pi$. Then  $\Sp_\nu \llp x_{i_0},\dots,x_{i_\ell}\rrp=\Sp_\nu \llp y_{i_0},\dots,y_{i_\ell}\rrp$ for any subset $\{i_0,\dots,i_\ell\}\subseteq\{0,\dots,n\}$. So we have reduced the lemma to the special case when $x_0,\dots,x_n\in \mathcal{O}$ and are not congruent modulo~$\pi$. 

In this case, we have $\Sp_\nu\llp x_0,\dots,x_n\rrp=\llp \overline{x}_0,\dots,\overline{x}_n\rrp$. 
We introduce the notation 
\[
T[i,j] \,\coloneqq\, \begin{aligned}[t] 
& \llp \overline{x}_{j}, \overline{x}_{j+1}, \dots, \overline{x}_{j+i}\rrp \wedge \llp \overline{x}_{j}, \overline{x}_{j+i+1}, \dots, \overline{x}_{j+n}\rrp \\
& {} - \Sp_\nu \llp x_{j}, x_{j+1}, \dots, x_{j+i}\rrp \wedge \Sp_\nu \llp x_{j}, x_{j+i+1}, \dots, x_{j+n}\rrp 
\end{aligned}
\]
and need to show that 
\begin{equation} \label{difference_correlators}
\sum_{j=0}^{n}\sum_{i=1}^{n-1}T[i,j]=0\,.
\end{equation}
If $x_j,x_{j+1},\dots, x_{j+i}$ are not all congruent modulo $\pi$ and $x_{j}, x_{j+i+1}, \dots, x_{j+n}$ are not  all congruent  modulo~$\pi$, then
\begin{align*}
\Sp_\nu\llp x_{j}, x_{j+1}, \dots, x_{j+i}\rrp &=\llp \overline{x}_{j}, \overline{x}_{j+1}, \dots, \overline{x}_{j+i}\rrp \,, \text{ and } \\
\Sp_\nu \llp x_{j}, x_{j+i+1}, \dots, x_{j+n}\rrp &= \llp \overline{x}_{j}, \overline{x}_{j+i+1}, \dots, \overline{x}_{j+n}\rrp \,,
\end{align*}
so $T[i,j]$  vanishes. Next, since $x_0,\dots,x_n\in \mathcal{O}$ and are not congruent modulo~$\pi$, it can not happen that  $x_j,x_{j+1},\dots, x_{j+i}$  and $x_{j}, x_{j+i+1}, \dots, x_{j+n}$ are congruent modulo $\pi$ at the same time. Finally, notice that the remaining terms in~\eqref{difference_correlators} can be grouped in pairs in the following way. Every element $\llp x_j, x_{j+1}, \dots, x_{j+i}\rrp$ with entries congruent modulo $\pi$ appears in two terms: $T[i,j]$ and \mbox{$T[n-i,j+i]$}. We have 
\begin{align*}
T[i,j] & \,=\, {} \begin{aligned}[t]
& \llp \overline{x}_j, \overline{x}_{j+1},\dots, \overline{x}_{j+i}\rrp\wedge \llp \overline{x}_j, \overline{x}_{j+i+1},\dots, \overline{x}_{j+n}\rrp \\
& {} - \Sp_\nu \llp x_j, x_{j+1},\dots, x_{j+i}\rrp \wedge \Sp_\nu \llp x_j, x_{j+i+1},\dots, x_{j+n}\rrp \,, \end{aligned} \\[1ex]
T[n-i,j+i] & \,=\, {} \begin{aligned}[t] 
& \llp \overline{x}_{j+i},  \overline{x}_{j+i+1},\dots, \overline{x}_{j+n}\rrp \wedge \llp \overline{x}_{j+i}, \overline{x}_{j},  \dots, \overline{x}_{j+i-1} \rrp \\
& {} - \Sp_\nu \llp x_{j+i},  x_{j+i+1},\dots, x_{j+n}\rrp \wedge \Sp_\nu\llp x_{j+i}, x_{j},  \dots, x_{j+i-1}\rrp \,. \end{aligned}
\end{align*}
Assume (say) that $x_j, x_{j+1}, \dots, x_{j+i}$ are all congruent to $a$ modulo $\pi$.  Then  
\[
\llp \overline{x}_j, \overline{x}_{j+1},\dots, \overline{x}_{j+i}\rrp=\llp \overline{x}_{j+i}, \overline{x}_{j},  \dots, \overline{x}_{j+i-1}\rrp=\llp a,\dots,a\rrp =0\,.
\] 
This assumption means that not all $x_j, x_{j+i+1},\dots, x_{j+n}$ are congruent modulo $\pi$, so 
 \[
 \Sp_\nu\llp x_j, x_{j+i+1},\dots, x_{j+n}\rrp = \llp \overline{x}_j, \overline{x}_{j+i+1},\dots,\overline{x}_{j+n}\rrp 
\]
and, similarly,
\[
\Sp_\nu \llp x_{j+i},  x_{j+i+1},\dots, x_{j+n}\rrp = \llp \overline{x}_{j+i},  \overline{x}_{j+i+1},\dots, \overline{x}_{j+n}\rrp\,.
\]
Finally, observe that $\overline{x}_j=\overline{x}_{j+i}=a$, so 
\[
\llp \overline{x}_j, \overline{x}_{j+i+1},\dots, \overline{x}_{j+n}\rrp=\llp \overline{x}_{j+i},  \overline{x}_{j+i+1},\dots, \overline{x}_{j+n}\rrp\,.
\]
So, $T[i,j]+T[n-i,j+i]=0$.  From here the statement follows.
\end{proof}

\subsection{Space of relations and the correlators}\label{SectionRelations}
We will define the space of relations $\mathcal{R}_n(F)\subseteq \mathcal{A}_n(F)$ for all infinite fields, inductively in $n$; the Lie coalgebra of formal multiple polylogarithms is then defined as the quotient
\begin{equation} \label{FormulaQuotient}
\mathcal{L}^\mathrm{f}_n(F)=\frac{\mathcal{A}_n(F)}{\mathcal{R}_n(F)}\,.
\end{equation}
The projection of $\llp x_0,x_1,\dots, x_n\rrp \in \mathcal{A}_n(F)$ to $\mathcal{L}^\mathrm{f}_n(F)$ is denoted by $\cor(x_0,x_1,\dots, x_n)$ and called {\it the correlator}. 

In weight one we put $\mathcal{R}_1(F)=0$. We have $\L_1^\mathrm{f}(F)\cong F^{\times}_\Q$ and $\cor(x_0,x_1)=\log(x_1-x_0)$. Next, assume that spaces $\mathcal{R}_i(F')$ are defined for $1\leq i<n$ and all infinite fields $F'$. Consider the field $F(t)$; let $\nu_a$ be the discrete valuation corresponding to $a\in F\cup\{\infty\}$. In this setting, we denote the specialization $\Sp_{\nu_a, t-a}$ by $\Sp_{t\to a}$ and $\Sp_{\nu_{\infty}, \frac{1}{t}}$ by $\Sp_{t\to \infty}$. The space $\mathcal{R}_n(F)$ is spanned by elements
\[
\Sp_{t \to 0}(R)- \Sp_{t \to 1}(R)\in \mathcal{A}_n(F)
\]
for elements $R\in \mathcal{A}_n(F(t))$ with cobracket equal 0 in $\bigwedge^2 \mathcal{L}^\mathrm{f}(F(t))$. Note that the weight $n$ component of $\bigwedge^2 \mathcal{L}^\mathrm{f}(F(t))$ is defined by the inductive assumption.
We define the  $n$-th graded component of the Lie coalgebra of polylogarithms as the quotient space in \eqref{FormulaQuotient}. 

\begin{remark}
Note that \( \Sp_{t\to1} - \Sp_{t\to0} \) in the above definition can be replaced with \( \Sp_{t\to b} - \Sp_{t\to a} \) with $a\ne b$, $a,b,\in F$, without loss of generality, by choosing the generator $a+t(b-a)$ for $F(t)$. We may also take $b=\infty$, by choosing $a+t/(t-1)$.
\end{remark}

To show that the cobracket $\delta$ descends to $\mathcal{L}^\mathrm{f}(F)$ we need the following lemma.

\begin{lemma}\label{LemmaSpecializationRelations} Suppose $F$ is an infinite field.  For any $s_0\in F$, and any $n\in \mathbb{N}$ we have that \[
\Sp_{s\to s_0}\bigl(\mathcal{R}_n(F(s))\bigr)\subseteq \mathcal{R}_n(F) \,.\]
\end{lemma}

\begin{proof}
We prove the statement by induction on $n$.  For the base case $n=1$, there is nothing to check as \( \mathcal{R}_1(F) = \mathcal{R}_1(F(s)) = 0 \), so suppose \( n \geq 2 \). For subspaces $U_1, U_2$ of a vector space $V$ over $\Q$ we denote by $U_1\wedge U_2$ the span of elements $u_1\wedge u_2$ for $u_1\in U_1$ and $u_2\in U_2$.  The vector space $\mathcal{R}_n(F(s))$ is spanned by elements $\Sp_{t\to 0}R-\Sp_{t\to  1} R$ for $R\in \mathcal{A}_n(F(s,t))$ such that $\delta(R)$ lies in 
\[
\sum_{k=1}^{n-1}\mathcal{A}_k(F(s,t))\wedge \mathcal{R}_{n-k}(F(s,t))\subseteq \bigwedge\nolimits^2 \mathcal{A}(F(s,t)).
\]
 We need to show that
\begin{equation}\label{FormualSpecializationLemmaResult}
\Sp_{s\to s_0}(\Sp_{t\to 0}R-\Sp_{t\to 1} R)\in \mathcal{R}_n(F)\,.
\end{equation}
By the induction hypothesis and \autoref{LemmaSpecializationCoproduct}, we conclude that
\[
\delta(\Sp_{t\to 0}R)=(\Sp_{t\to 0}\wedge \Sp_{t\to 0})(\delta(R))\in 
\sum_{k=1}^{n-1}\mathcal{A}_k(F(s))\wedge \mathcal{R}_{n-k}(F(s));
\]
a similar statement holds for $\delta(\Sp_{t\to 1} R)$.
Thus for any $s_1\in F$ we have 
\begin{equation}\label{FormulaSpecialzationInLemma}
\Sp_{s\to s_0}(\Sp_{t\to 0}R-\Sp_{t\to 1} R)-\Sp_{s\to s_1}(\Sp_{t\to 0}R-\Sp_{t\to 1} R)\in \mathcal{R}_n(F).
\end{equation}

Assume that 
\[
R=\sum_{i=1}^N n_i \llp f_0^i(s,t),\dots,f_n^i(s,t)\rrp  \,, \quad \text{with $f_{j}^i(s,t) \in F(s,t) $, $ n_i \in \mathbb{Q}$} \,,
\]
and write $R_i=\llp f_0^i(s,t),\dots,f_n^i(s,t)\rrp$ for the $i$-th summand.
Call a point $(s_1,t_1)\in \mathbb{A}^2$ {\it regular} for~$R$ if for every $i$ the rational map
$\mathbb{A}^2\longrightarrow \mathbb{P}^{n-1}$ sending $(s,t)$ to $[f_1^i(s,t)-f_0^i(s,t):\dots:f_n^i(s,t)-f_0^i(s,t)]$ is regular at $(s_1,t_1)$. The number of points at which $R$ is not regular is finite. For a regular point $(s_1,t_1)\in \mathbb{A}^2$ we can find polynomials  $A_1,\dots, A_n\in F[s,t]$ such that
\[
[f_1^i(s,t)-f_0^i(s,t):\dots:f_n^i(s,t)-f_0^i(s,t)]=[A_1(s,t):\dots:A_n(s,t)]
\]
and not all $A_j(s_1,t_1)$ vanish. Properties \ref{Atranslate} and \ref{Ascale} imply that
\[
\Sp_{s\to s_1} \Sp_{t\to t_1}R_i = \Sp_{t\to t_1} \Sp_{s\to s_1}R_i = \llp 0,A_1(s_1,t_1),\dots,A_n(s_1,t_1)\rrp \,.
\]

Since $F$ is infinite, we can find a point $s_1\in F$  such that points $(s_1,0)$ and $(s_1,1)$ are regular for $R$. By the induction hypothesis and \autoref{LemmaSpecializationCoproduct}, we have $\delta\Sp_{s\to s_1}(R)=0$, so
\[
\Sp_{s\to s_1}(\Sp_{t\to 0}(R)-\Sp_{t\to 1} (R))=\Sp_{t\to 0}(\Sp_{s\to s_1}(R))-\Sp_{t\to 1}(\Sp_{s\to s_1}(R))\in \mathcal{R}_n(F)\,.
\]
Thus,~\eqref{FormulaSpecialzationInLemma} implies \eqref{FormualSpecializationLemmaResult}. This finishes the proof of the Lemma.
\end{proof}

\begin{corollary}\label{LemmaRCoideal} The subspace $\mathcal{R}(F)$ is a Lie coideal in  $\mathcal{A}(F)$. 
\end{corollary}
\begin{proof} 

By construction, $\mathcal{R}_n(F)$ is spanned by $\Sp_{t\to 0}R-\Sp_{t\to 1} R$ for $R\in \mathcal{A}_n(F(t))$ satisfying $\delta(R)\in \sum_{k=1}^{n-1} \mathcal{A}_k(F(t))\wedge \mathcal{R}_{n-k}(F(t))$. By \autoref{LemmaSpecializationCoproduct} and \autoref{LemmaSpecializationRelations}, we have 
\begin{align*}
\delta(\Sp_{t\to 0}R-\Sp_{t\to 1} R) & {} = (\Sp_{t\to 0}\otimes \Sp_{t\to 0})(\delta R)-(\Sp_{t\to 1} \otimes \Sp_{t\to 1}) (\delta R) \\
& {} \in \sum_{k=1}^{n-1} \mathcal{A}_k(F)\wedge \mathcal{R}_{n-k}(F)\,.
\end{align*}
This shows that  $\mathcal{R}(F)$ is a Lie coideal in $\mathcal{A}(F)$.
\end{proof} 

The Lie coalgebra of formal multiple polylogarithms is then the aforementioned quotient
\[
\mathcal{L}^\mathrm{f}_n(F)=\frac{\mathcal{A}_n(F)}{\mathcal{R}_n(F)}\,.
\]
It is spanned by correlators $\cor(x_0,\dots, x_n)$, whose cobracket is given by the formula
\begin{equation}\label{FormulaCoproductCorrelators}
\delta \cor(x_0,\dots, x_n)=\sum_{j=0}^n\sum_{i=1}^{n-1} \cor(x_{j}, x_{j+1}, \dots, x_{j+i})\wedge  \cor(x_{j},  x_{j+i+1}, \dots, x_{j+n})\,.
\end{equation}

\begin{corollary} Let $F$ be an infinite field and $s_0\in F$.
The specialization map $\Sp_{s\to s_0}\colon \mathcal{A}(F(s))\lra \mathcal{A}(F)$ descends to a well-defined morphism of Lie coalgebras
\begin{align}\label{FormulaSpecializationMap}
\Sp_{s\to s_0}\colon \mathcal{L}^\mathrm{f}(F(s))\longrightarrow \mathcal{L}^\mathrm{f}(F)\,.
\end{align}
\end{corollary}
\begin{proof}
The statement follows from \autoref{LemmaSpecializationRelations} and \autoref{LemmaRCoideal}.
\end{proof}

\begin{question} Consider an arbitrary discrete valuation $\nu$ on $F$, with residue field $\mathrm{k}$. Is the  specialization map $\Sp_{\nu,\pi}\colon \mathcal{L}^\mathrm{f}_n(F)\lra \mathcal{L}^\mathrm{f}_n(\mathrm{k})$ well defined?
\end{question}

\begin{corollary} Let $F$ be an infinite field. The natural map
\[
i\colon \mathcal{L}^\mathrm{f}(F)\longrightarrow \mathcal{L}^\mathrm{f}(F(s))
\]
is injective and induces an isomorphism
\begin{equation}\label{FormulaMapH1}
i\colon H^1(\mathcal{L}^\mathrm{f}(F),\Q)_n \longrightarrow H^1(\mathcal{L}^\mathrm{f}(F(s)),\Q)_n\,, \qquad n \geq 2\,.
\end{equation}
\end{corollary}
\begin{proof}
The map $\Sp_{s \to 0}\colon \mathcal{L}^\mathrm{f}(F(s)) \longrightarrow  \mathcal{L}^\mathrm{f}(F)$ satisfies an identity ${\Sp_{s \to 0}} \circ i = \Id$, so $i$ is injective. Next, $H^1(\mathcal{L}^\mathrm{f}(F),\Q)=\ker\bigl(\delta \colon \mathcal{L}^\mathrm{f}(F) \longrightarrow \bigwedge^2 \mathcal{L}^\mathrm{f}(F) \bigr)$, and similarly for $H^1(\mathcal{L}^\mathrm{f}(F(s)),\Q)$. It follows that the map  \eqref{FormulaMapH1} is injective. To see that it is surjective, take any element $R\in \mathcal{L}^\mathrm{f}(F(s))$ such that $\delta(R)=0$. Consider the element $R-R'\in \mathcal{L}^\mathrm{f}(F(s,t))$ where $R'$ is the image of $R$ under the automorphism of  $F(s,t)$ exchanging $s$ and $t$. We have $\delta(R-R')=0$, so $\Sp_{t\to s}(R-R')=\Sp_{t\to 0}(R-R')$. It follows that 
\[
R=\Sp_{t\to 0}R'=\Sp_{s\to 0}R\in \Im(i)\,. \qedhere
\]
\end{proof}

\subsection{Properties of correlators} \label{SectionCorrelatorProperties}

With this setup we now show some basic properties of correlators.

\begin{proposition}[Distribution relations] \label{LemmaDistributionCor} Let $N\in \mathbb{N}$ be such that $(N,\ch(F))=1$. Assume that $F$ contains all $N\!$-th roots of unity. Then we have
\begin{equation}\label{FormulaDistributionCor}
\cor(x_0^N,\dots,x_n^N)=\sum_{\zeta_1^N=1,\dots,\zeta_n^N=1}
\cor(x_0,\zeta_1 x_1,\dots,\zeta_{n}x_n)\,.
\end{equation}
\end{proposition}
\begin{proof}
We argue by induction on $n$. For $n=1$ the claim follows from the identity $x^N-y^N={\prod_{\zeta^N=1}(x-\zeta y)}$. For $n>1$ the induction assumption easily implies that
\[
\delta\Big(\cor(t^Nx_0^N,x_1^N,\dots,x_n^N)-\sum_{\zeta_1^N=1,\dots,\zeta_n^N=1}\cor(tx_0,\zeta_1x_1,\dots,\zeta_{n}x_n)\Big)=0\,.
\]
We thus get an element \( R = \cor(t^Nx_0^N,x_1^N,\dots,x_n^N)-\sum_{\zeta_1^N=1,\dots,\zeta_n^N=1}\cor(tx_0,\zeta_1x_1,\dots,\zeta_{n}x_n) \) which is in the kernel of \( \delta \) mapping from \( \mathcal{L}^\mathrm{f}_n(F(t)) \).
Using \ref{Acycle} if needed, we may assume that $x_0\ne 0$, and consider $\Sp_{t\to 1}R-\Sp_{t\to\infty}R\in \mathcal{R}_n(F)$. We have 
\[\Sp_{t\to \infty}R=\cor(x_0^N,0,\dots,0)-\sum_{\zeta_1^N=1,\dots,\zeta_n^N=1}\cor(x_0,0,\dots,0)=0\]
by \ref{Aone} and \ref{Ascale}. Therefore, 
$\Sp_{t\to 1}R\in \mathcal{R}_n(F)$, which is exactly the claimed 
identity.
\end{proof}

A similar proof gives the following.
\begin{proposition} \label{PropositionFrobenius}
Assume that $F$ is a field of characteristic $p$. Then  
\[\cor(x_0^p,\dots,x_n^p) = p^n\cor(x_0,\dots,x_n).\]
\end{proposition}

Although not imposed as a relation in \autoref{SectionAn}, one can show that correlators satisfy a reversal symmetry. Hence overall the correlators are dihedrally symmetric.

\begin{proposition}\label{lem:correv}
    The following reversal symmetry holds, for \( n \geq 1 \),
    \begin{equation}\label{eqn:correv}
        \cor(x_0,x_1,\ldots,x_{n-1}, x_n) = (-1)^{n+1} \cor(x_n,x_{n-1}, \ldots, x_1, x_0) \,.
    \end{equation}
\end{proposition}
\begin{proof}
    We argue by induction on \( n \).  For \( n = 1 \), this is trivial by the cyclic symmetry of the correlators, namely \( \cor(x_0,x_1) = \cor(x_1,x_0) \).  For \( n > 1 \), we compute using~\eqref{FormulaCoproductCorrelators}
    \begin{align*}
    \delta \cor(x_n,\dots, x_0)
    &=\sum_{i=1}^{n-1}\sum_{j=0}^n \cor(x_{n-j}, x_{n-j-1}, \dots, x_{n-j-i})\wedge  \cor(x_{n-j},  x_{n-j-i-1}, \dots, x_{n-j-n})\\
    &=(-1)^n\sum_{i=1}^{n-1}\sum_{j=0}^n \cor(x_{n-j-i}, \dots, x_{n-j-1}, x_{n-j})\wedge  \cor(x_{n-j-n}, \dots, x_{n-j-i-1}, x_{n-j})\\
    &=(-1)^n\sum_{i'=1}^{n-1}\sum_{j'=0}^n \cor(x_{j'-(n-i')}, \dots, x_{j'-1}, x_{j'})\wedge  \cor(x_{j'-n}, \dots, x_{j'-(n-i')-1}, x_{j'})\\
    &=(-1)^n\sum_{i'=1}^{n-1}\sum_{j'=0}^n \cor(x_{j'+i'+1}, \dots, x_{j'+n}, x_{j'})\wedge  \cor(x_{j'+1}, \dots, x_{j'+i'}, x_{j'})\\
    &=(-1)^{n+1} \delta \cor(x_0,\dots, x_n)\,,
    \end{align*}
    where in the second equality we used the induction assumption, in the second we changed summation indices to $j'=n-j$ and $i'=n-i$, in the third we shifted all indices by $(n+1)$, and in the last equality we interchanged the wedge factors and used cyclic symmetry. Therefore,
    \[
        \delta \Big(  \cor(x_0 t, x_1,\ldots,x_{n-1}, x_n) - (-1)^{n+1} \cor(x_n, x_{n-1},\ldots,x_1,x_0 t) \Big) = 0 \,.
    \]
    We thus get an element \( R = \cor(x_0 t, x_1,\ldots,x_{n-1}, x_n) - (-1)^{n+1} \cor(x_n, x_{n-1},\ldots,x_1,x_0 t) \), which is in the kernel of \( \delta \) mapping from \( \mathcal{L}^\mathrm{f}_n(F(t)) \).  Computing the specialization \( \Sp_{t\to1} R - \Sp_{t\to\infty} R \) gives the claimed identity.
\end{proof}

Finally, we have the shuffle relations for correlators. 

\begin{proposition}[Shuffle relations]\label{LemmaShuffleCor} For \( n_1, n_2 \geq 1 \), the following relation holds, 
\begin{equation}\label{FormulaShuffleCor}
\sum_{\sigma \in \Sigma_{n_1,n_2}} \cor(x_0,x_{\sigma(1)},\dots,x_{\sigma(n_1+n_2)})=0\,,
\end{equation}
where $\Sigma_{n_1,n_2}\subseteq \mathfrak{S}_{n_1+n_2}$ is the set of $(n_1,n_2)$-shuffles, i.e., permutations $\sigma \in \mathfrak{S}_{n_1+n_2}$ such that $\sigma^{-1}(1)<\dots<\sigma^{-1}(n_1)$ and $\sigma^{-1}(n_1+1)<\dots<\sigma^{-1}(n_1+n_2)$.
\end{proposition}
\begin{proof}
As in \autoref{lem:correv} the proof goes by induction on $n=n_1+n_2$. The base of induction $n=2$ follows from . For the combinatorial part, namely the fact that~\eqref{FormulaShuffleCor} is annihilated by $\delta$, see the proof of Theorem 4.3 in \cite[p.437--438]{Gon01B}.
\end{proof}

\section{Construction of \texorpdfstring{$\mathcal{H}^\mathrm{f}(F)$}{H\textasciicircum{}f\_bullet(F)}}
In this section we construct the Hopf algebra $\mathcal{H}^\mathrm{f}(F)$. This Hopf algebra is defined abstractly as a universal coenveloping coalgebra of $\mathcal{L}^\mathrm{f}(F)$, but it also comes with a distinguished collection of generators, the iterated integrals, and we will begin by defining these elements on the level of the Lie coalgebra $\mathcal{L}^\mathrm{f}(F)$ using correlators.

\subsection{Iterated integrals}\label{sec:IteratedIntegrals}

We start by defining iterated integrals via correlators, motivated by~\eqref{FormulaIntegralsCorrelators}, and show how to compute their coproduct.  

\begin{definition} \label{def:IteratedIntegralsViaCorrelators}
Consider $x_0,\dots,x_{n+1} \in F$ for $n\geq 0$. We define the iterated integral, as an element in the Lie coalgebra $\mathcal{L}^\mathrm{f}(F)$, by the formula:
\[
			\IL(x_0; x_1,\ldots,x_n; x_{n+1}) \coloneqq  \cor(x_1,\ldots,x_n,x_{n+1}) - \cor(x_0, x_1,\ldots, x_n) \in \mathcal{L}^\mathrm{f}_{n}(F) \,.
\]
By convention, we set \( \IL(x_0; x_1) = 0 \).

\end{definition}

 We can now give a precise meaning to the informal viewpoint that correlators are iterated integrals with lower bound equal to \( \infty \).

\begin{corollary} \label{cor:IisCorinf}
    The following identity holds
    \[
        \Sp_{x_0 \to \infty} \IL(x_0; x_1,\ldots,x_n; x_{n+1}) = \cor(x_1,\ldots,x_n, x_{n+1}) \,.
    \]
\begin{proof}
By definition of specialization $\Sp_{x_0\to\infty}\cor(x_0,x_1,\dots,x_n) = \cor(1,0,\dots,0)=0$, and hence the claim follows.
\end{proof}
\end{corollary}

\begin{proposition}\label{prop:corasint}
The correlator can be expressed via iterated integrals as follows
\begin{align*}
    \cor(x_0, x_1,\ldots,x_n) 
    & = \sum_{i=0}^{n} \IL(0; \overbrace{0,\ldots,0}^{i}, x_0, x_1,\ldots,x_{n-1-i};x_{n-i}) \\
    & = \IL(0; x_0, \ldots, x_{n-1}; x_n) + \IL(0; 0, x_0, \ldots, x_{n-2}; x_{n-1}) \\
    & \quad +  \IL(0; 0, 0, x_0, \ldots, x_{n-3}; x_{n-2}) + \cdots + \IL(0; 0,\ldots,0;x_0) \,.
\end{align*}
\begin{proof}
    When expressed in terms of correlators using \autoref{def:IteratedIntegralsViaCorrelators} the terms on the right-hand side cancel in pairs, leaving \( \cor(x_0,x_1,\ldots,x_n) - \cor(0,\ldots,0) = \cor(x_0,x_1,\ldots,x_n)\) by \ref{Azero}.
\end{proof}
\end{proposition}

Using the formula for the cobracket of a correlator given in \eqref{FormulaCoproductCorrelators}, one can derive a formula for the cobracket of the iterated integral, as defined above.

\begin{proposition}\label{lem:intascor}
		The following equality holds
\begin{align*}
&\delta \IL(x_0; x_1,\ldots,x_n; x_{n+1}) 
				&  = \!\!\! \sum_{0 \leq i < j \leq n+1} \!\!\! \IL(x_0; x_1,\ldots,x_i, x_j,\ldots, x_n; x_{n+1}) \wedge \IL(x_i; x_{i+1}, \ldots, x_{j-1}; x_j)  \,.
\end{align*}

\begin{proof} By expanding out the right hand side, using \autoref{def:IteratedIntegralsViaCorrelators}, we have
    \begin{align*}
    &\sum_{0 \leq i < j \leq n+1} \IL(x_0; x_1,\ldots,x_i, x_j,\ldots, x_n; x_{n+1}) \wedge \IL(x_i; x_{i+1}, \ldots, x_{j-1}; x_j)  \\			
       &  = \sum_{0 \leq i < j \leq n+1} \begin{aligned}[t] 
            \big(\cor&(x_1,\ldots,x_i, x_j,\ldots, x_n, x_{n+1}) - \cor(x_0, x_1,\ldots,x_i, x_j,\ldots, x_n)\big) \\
            & \wedge \big(\cor(x_{i+1}, \ldots, x_{j-1}, x_j) - \cor(x_i, x_{i+1}, \ldots, x_{j-1}) \big) \,. \end{aligned}
    \end{align*}
    (For $j=n+1$ we interpret $\cor(x_0, x_1,\ldots,x_i, x_j,\ldots, x_n)$ as $\cor(x_0, x_1,\ldots,x_i)$ and similarly for the term $\cor(x_1,\ldots,x_i, x_j,\ldots, x_n, x_{n+1})$ when $i=0$.)
    Rewrite the last sum as follows
    {
    \begin{align*}
        & = \begin{aligned}[t]
        & \sum_{1 \leq p < q \leq n} A_{p,q} \wedge \cor(x_p, x_{p+1},\ldots, x_{q}) \\[-0.5ex]
        & {} + \sum_{1 \leq q \leq n} \big( {-}\cor(x_q, x_{q+1}, \ldots, x_{n+1}) + \cor(x_0, x_q, x_{q+1},\ldots,x_n) \big) \wedge \cor(x_0, x_1,\ldots,x_q)  \\[-0.5ex]
        & {} + \sum_{1 \leq p \leq n} \big( \cor(x_1, x_2, \ldots, x_p, x_{n+1}) - \cor(x_0, x_1, \ldots,x_p) \big) \wedge \cor(x_p, x_{p+1},\ldots,x_{n+1}) 
        \end{aligned} \\[2ex]
         & = \begin{aligned}[t]
        & \sum_{1 \leq p < q \leq n} A_{p,q} \wedge \cor(x_p, x_{p+1},\ldots, x_{q}) \\[-0.5ex]
        & {} + \sum_{1 \leq q \leq n} \cor(x_0, x_{q}, x_{q+1}, \ldots, x_{n}) \wedge \cor(x_0, x_1,\ldots,x_q)  \\[-0.5ex]
        & {} + \sum_{1 \leq p \leq n} \cor(x_1, x_2, \ldots, x_p, x_{n+1}) \wedge \cor(x_p, x_{p+1},\ldots,x_{n+1}) \,, \end{aligned}
    \end{align*}
    }
    where
    \begin{align*}
        A_{p,q} = {} & \cor(x_1,\ldots,x_{p-1},x_q,\ldots,x_{n+1})- \cor(x_0,\ldots,x_{p-1},x_q,\ldots,x_{n}) \\
        &  - \cor(x_1,\ldots,x_{p},x_{q+1},\ldots,x_{n+1})+ \cor(x_0,\ldots,x_{p},x_{q+1},\ldots,x_{n}) \,.
    \end{align*}

    Next, we compute the LHS, using \eqref{FormulaCoproductCorrelators}. For any $y_0,\dots,y_n$ we have (following the convention that $y_{i+n+1}=y_i$)
    \begin{align*}
    & \delta \cor(y_0,y_1,\ldots,y_{n}) \\
     & = \sum_{j=0}^n\sum_{i=1}^{n-1} \cor(y_{j}, y_{j+1}, \dots, y_{j+i})\wedge  \cor(y_{j},  y_{j+i+1}, \dots, y_{j+n}) \\
    & = \begin{aligned}[t]
    & \sum_{1 \leq p < q \leq n} 
    \!\! \big(  \cor(y_0,\ldots,y_{p-1},y_q,\ldots,y_{n})  - \cor(y_0,\ldots,y_p,y_{q+1},\ldots,y_{n}) \big) 
     \wedge \cor(y_p, y_{p+1}, \ldots, y_q) \\
    & + \sum_{1 \leq p \leq n} \!\! \cor(y_{0}, y_1,y_2,\ldots,y_p) \wedge \cor(y_{0}, y_p, y_{p+1},\ldots,y_n) \,,
    \end{aligned}
    \end{align*}
    We therefore obtain
    \begin{align*}
      \delta \cor(x_0,x_1,\ldots,x_{n}) 
     & = \begin{aligned}[t]
     & \sum_{1 \leq p < q \leq n} A'_{p,q} \wedge \cor(x_p, x_{p+1}, \ldots, x_q) \\
    & + \sum_{1 \leq p \leq n} \cor(x_{0}, x_1,x_2,\ldots,x_p) \wedge \cor(x_{0}, x_p, x_{p+1},\ldots,x_n) \,,
    \end{aligned}
    \end{align*}
    \begin{align*}
         \delta \cor(x_1,x_2,\ldots,x_{n+1}) 
        & = \begin{aligned}[t]
            & \sum_{1 \leq p < q \leq n} A''_{p,q} \wedge \cor(x_p, x_{p+1}, \ldots, x_q) \\
            & + \sum_{1 \leq p \leq n} \cor(x_{n+1}, x_1,x_2,\ldots,x_p) \wedge \cor(x_{n+1}, x_p, x_{p+1},\ldots,x_n)  \,,
            \end{aligned}
    \end{align*}
    where 
    \begin{align*}
        A'_{p,q} &= \cor(x_0,\ldots,x_{p-1},x_q,\ldots,x_{n}) -\cor(x_0,\ldots,x_p,x_{q+1},\ldots,x_{n}) \,, \\
        A''_{p,q} &= \cor(x_1,\ldots,x_{p-1},x_q,\ldots,x_{n+1}) -\cor(x_1,\ldots,x_p,x_{q+1},\ldots,x_{n+1})  \,.
    \end{align*}
    Since \( A_{p,q} = A'_{p,q} + A''_{p,q} \), we see that the LHS equals to the RHS.
\end{proof} 
\end{proposition}

\subsection{Properties of iterated integrals} \label{sec:propertiesII}

A number of the properties of correlators established in \autoref{SectionCorrelatorProperties}, or encoded as part of their axioms in \refAall, have corresponding results on the level of iterated integrals.

\begin{proposition}[Affine invariance of \( \IL \)]\label{cor:int:affine}
    For $n \geq 2$, $a \in F^\times$, and $x_0,\dots,x_{n+1}, b \in F$ we have
    \begin{equation}\label{eqn:int:affine}
        \IL(x_0; x_1, \ldots, x_n; x_{n+1}) = \IL(ax_0 + b; ax_1 + b, \ldots, ax_n + b; ax_{n+1} + b) \,.
    \end{equation}
\end{proposition}
\begin{proof}
    This follows directly from properties \ref{Atranslate} and \ref{Ascale} for correlators, after writing
    \[
        \IL(x_0; x_1, \ldots, x_n; x_{n+1}) =  \cor(x_1,\ldots, x_n, x_{n+1}) - \cor(x_0, x_1,\ldots, x_n) \,. \qedhere
    \]
\end{proof}

\begin{proposition}\label{PropShuffleIter}
For any $a,b,x_1,\dots,x_{n_1+n_2} \in F$, $n_1,n_2\ge1$, we have
\[\sum_{\sigma \in \Sigma_{n_1,n_2}} \IL(a;x_{\sigma(1)},\dots,x_{\sigma(n_1+n_2)};b) = 0\,,\]
where $\Sigma_{n_1,n_2}\subseteq \mathfrak{S}_{n_1+n_2}$ is the set of $(n_1,n_2)$-shuffles, i.e., permutations $\sigma$ satisfying $\sigma^{-1}(1)<\dots<\sigma^{-1}(n_1)$ and $\sigma^{-1}(n_1+1)<\dots<\sigma^{-1}(n_1+n_2)$.   
\end{proposition}
\begin{proof}
When written in terms of correlators this follows from~\eqref{FormulaShuffleCor} and the cyclic symmetry \ref{Acycle}.
\end{proof}
	
\begin{proposition}\label{PropDihedralIter}
The following reversal symmetry holds
    \begin{equation}\label{eqn:revsym}
    \IL(x_0; x_1,\ldots,x_n; x_{n+1}) = (-1)^{n+1} \IL(x_0; x_n,\ldots,x_1; x_{n+1}) \,.
    \end{equation}
\end{proposition}
\begin{proof}   
Likewise, this follows from \eqref{eqn:correv} and the cyclic symmetry \ref{Acycle}.
\end{proof}

The distribution relations also hold for \( \IL \); see \autoref{LemmaDistributionInt} below.
There is also a cyclic symmetry for iterated integrals, but it is more involved, see \autoref{cor:int:dihedral} below.

\subsection{The Hopf algebra of formal multiple polylogarithms}\label{SectionHopfAlebra} 
Let $Q$ be the functor from the category of commutative graded connected Hopf algebras to the category of graded Lie coalgebras sending a graded connected Hopf algebra $H$ to its graded Lie coalgebra of indecomposables $Q(H)=H_{>0}/(H_{>0} \cdot H_{>0})$.  This functor 
 has a right adjoint functor $U^c$ sending a graded Lie coalgebra $\L$ to its universal coenveloping coalgebra  $U^c(\L)$, see \cite[\S 3]{Mic80}.

\begin{definition} Let $F$ be an infinite field. The graded Hopf algebra of formal multiple polylogarithms $\mathcal{H}^\mathrm{f}(F)$ is the universal coenveloping coalgebra $U^c(\mathcal{L}^\mathrm{f}(F))$ of $\mathcal{L}^\mathrm{f}(F)$. 
\end{definition}

The Lie coalgebra $\mathcal{L}^\mathrm{f}(F)$ is positively graded, so $\mathcal{H}^\mathrm{f}(F)$ is a commutative graded connected Hopf algebra, which is free as a commutative algebra and $Q(\mathcal{H}^\mathrm{f}(F))=\mathcal{L}^\mathrm{f}(F)$, see \cite{MM65, Car07}. Denote by $p\colon \mathcal{H}^\mathrm{f}(F) \longrightarrow  \mathcal{L}^\mathrm{f}(F)$ the projection onto the space of indecomposables.

In the next lemma we describe elements in $\mathcal{H}^\mathrm{f}(F)$ called iterated integrals.

\begin{lemma} \label{PropertiesIteratedIntegrals}
There exists a unique collection of elements $\I(x_0;x_1,\dots,x_n;x_{n+1})\in \mathcal{H}^\mathrm{f}_n(F)$ for $n\geq 0$ and $x_0,\dots,x_{n+1}\in F$ with coproduct
\begin{align*}
&\Delta(\I(x_0;x_1,\dots,x_n;x_{n+1}))=\\
&\sum_{\substack{0=i_0<i_1<\dots \\ \dots<i_k<i_{k+1}=n+1}}\I(x_{i_0};x_{i_1},\dots,x_{i_k};x_{i_{k+1}})\otimes \prod_{j=0}^k\I(x_{i_j};x_{i_{j}+1},\dots,x_{i_{j+1}-1};x_{i_{j+1}})\,.
\end{align*}
such that 
\begin{equation}\label{FormulaCorIt}
p(\I(x_0;x_1,\dots,x_n;x_{n+1}))=\IL(x_0; x_1,\ldots,x_n; x_{n+1}) \in \L_{n}(F) \text{ for $n\geq 1$}\,,
\end{equation}
and the following properties hold:
\begin{enumerate}[label=\rm(P\arabic*),ref=\rm(P\arabic*)]
\item\label{rel:unit} $\I(x_0;x_1)=1\in \mathcal{H}^\mathrm{f}_0(F)$
\item\label{rel:shuffle}  $\I(a;x_1,\dots,x_n;b)   \I(a;x_{n_1+1},\dots,x_{n_1+n_2};b)=\sum_{\sigma \in \Sigma_{n_1,n_2}} \I(a;x_{\sigma(1)},\dots,x_{\sigma(n_1+n_2)};b)$,\\
where $\Sigma_{n_1,n_2}\subseteq \mathfrak{S}_{n_1+n_2}$ is the set of $(n_1,n_2)$-shuffles, i.e., permutations $\sigma \in \mathfrak{S}_{n_1+n_2}$ such that $\sigma^{-1}(1)<\dots<\sigma^{-1}(n_1)$ and $\sigma^{-1}(n_1+1)<\dots<\sigma^{-1}(n_1+n_2)$.   
\item \label{rel:pathdecomp} $\I(x_0;x_1,\dots,x_n;x_{n+1})=\sum_{k=0}^n\I(x_0;x_1,\dots,x_k;a)\I(a;x_{k+1},\dots,x_n;x_{n+1})$ for $n\geq 0$ and $a\in F$,
\item\label{rel:equalbounds} $\I(x_0;x_1,\dots,x_n;x_{n+1})=0$ if $x_0=x_{n+1}$ and $n\geq 1$.
\end{enumerate}
\end{lemma}
\begin{proof}  Consider the Hopf algebra $\mathcal{I}(S)$ defined by Goncharov in \cite[\S 2.1]{Gon05} and take $S=F$. It is spanned by elements 
\[
\mathbb{I}(x_0;x_1,\dots,x_n;x_{n+1}) \text{ for } x_0,\dots,x_{n+1}\in F, n\geq 0 \,,
\]
subject to the relations
\begin{enumerate}[label=\rm(R\arabic*),ref=\rm(R\arabic*)]
\item\label{int:rel1} The unit: $\mathbb{I}(x_0;x_1)=1\in \mathcal{I}_{0}(F)$,
\item\label{int:rel2} The shuffle product formula: for $m, n\geq 0$ one has  
\[
\mathbb{I}(a;x_1,\dots,x_n;b)   \mathbb{I}(a;x_{n+1},\dots,x_{n+m};b)=\sum_{\sigma \in \Sigma_{n,m}} \mathbb{I}(a;x_{\sigma(1)},\dots,x_{\sigma(n+m)};b)\,,
\]
\item\label{int:rel3} The path composition formula: for any $n\geq 0$ and $a\in F$ one has 
\[
\mathbb{I}(x_0;x_1,\dots,x_n;x_{n+1})=\sum_{k=0}^n\mathbb{I}(x_0;x_1,\dots,x_k;a)\mathbb{I}(a;x_{k+1},\dots,x_n;x_{n+1}) \,,
\]
\item\label{int:rel4}  $\mathbb{I}(x_0;x_1,\dots,x_n;x_{n+1})=0$ if $x_0=x_{n+1}$ and $n\geq 1$.
\end{enumerate}
The coproduct in $\mathcal{I}(S)$ is given by the formula
\begin{align*}
&\Delta(\mathbb{I}(x_0;x_1,\dots,x_n;x_{n+1}))=\\
&\sum_{\substack{0=i_0<i_1<\dots \\ \dots <i_k<i_{k+1}=n+1}}\mathbb{I}(x_{i_0};x_{i_1},\dots,x_{i_k};x_{i_{k+1}})\otimes \prod_{j=0}^k\mathbb{I}(x_{i_j};x_{i_{j}+1},\dots,x_{i_{j+1}-1};x_{i_{j+1}})\,.
\end{align*}

Consider the map 
\[
 \ell \colon \mathcal{I}(F) \longrightarrow \mathcal{L}^\mathrm{f}(F)
\]
sending products $\mathcal{I}_{>0}\cdot\mathcal{I}_{>0}$ to $0$ and such that 
\[
\ell\bigl(\mathbb{I}(x_0;x_1,\dots,x_n;x_{n+1})\bigr)=\IL(x_0; x_1,\ldots,x_n; x_{n+1})\,.
\]
To see that it is well-defined, we need to check that 
relations \ref{int:rel1}--\ref{int:rel4}
above hold in $\mathcal{L}^\mathrm{f}(F)$.  Relations \ref{int:rel1}, \ref{int:rel3}, and \ref{int:rel4} are obvious and \ref{int:rel2} follows from \autoref{PropShuffleIter}. Since $\ell$ kills products, it induces a map $\ell \colon  Q(\mathcal{I})\longrightarrow\mathcal{L}^\mathrm{f}(F)$. By \autoref{lem:intascor}, this is a morphism of Lie coalgebras ($\ell$ maps to zero any term in the coproduct $\Delta(\I(x_0;x_1,\dots,x_n;x_{n+1}))$ than involves a non-trivial product of iterated integrals; the remaining terms are exactly the ones appearing in \autoref{lem:intascor}). By the adjunction between functors $U^c$ and $Q$, we have a morphism of Hopf algebras $\mathcal{I}(F)\longrightarrow U^c(\mathcal{L}^\mathrm{f}(F))$. The images of $\mathbb{I}(x_0;x_1,\dots,x_n;x_{n+1})$ satisfy all the properties we asked for.

Next, we show uniqueness. Assume that we have two families of elements $\I$ and $\I'$ satisfying the properties above. We prove that for any $x_0,\dots x_{n+1}\in F$ we have $\I(x_0;x_1,\dots,x_n;x_{n+1})=\I'(x_0;x_1,\dots,x_n;x_{n+1})$ by induction on $n$. For $n=0$ it follows from \ref{rel:unit}. For $n=1$ it follows from \eqref{FormulaCorIt} and the fact that $\mathcal{H}^\mathrm{f}_1(F)\cong\L_1(F)$. Assume that $n\geq 2$. By induction, for any $x_0,\dots x_{n+1}\in F$  we have 
\[
\Delta'\bigl(\I(x_0;x_1,\dots,x_n;x_{n+1})-\I'(x_0;x_1,\dots,x_n;x_{n+1})\bigr)=0\,,
\]
where \( \Delta' = \Delta - 1 \otimes \id - \id \otimes 1 \) is the reduced coproduct.
On the other hand, by  \eqref{FormulaCorIt} we have 
\[
p\bigl(\I(x_0;x_1,\dots,x_n;x_{n+1})-\I'(x_0;x_1,\dots,x_n;x_{n+1})\bigr)=0\,.
\]
The statement then follows from the fact that $p$ induces an isomorphism between the set of primitive elements in $\mathcal{H}^\mathrm{f}(F)$ and $\ker(\delta)\subseteq \mathcal{L}^\mathrm{f}(F)$, see \cite{MM65}.
\end{proof}

\begin{proposition} Let $F$ be an infinite field. Iterated integrals $\I(x_0;x_1,\dots,x_n;x_{n+1})\in \mathcal{H}^\mathrm{f}_n(F)$ for $n\geq 0$ and $x_0,\dots,x_{n+1}\in F$ span  $\mathcal{H}^\mathrm{f}(F)$ as a $\Q$-vector space.
\end{proposition}
\begin{proof}
We prove that iterated integrals of weight $n$ span  $\mathcal{H}^\mathrm{f}_n(F)$ by induction on $n.$ The cases $n=0,1$ are obvious. For $n\geq 2,$ consider any element $a\in \mathcal{H}^\mathrm{f}_n(F).$ By \autoref{prop:corasint}, there exists a linear combination $a'$ of iterated integrals such that $p(a)=p(a')\in \mathcal{L}^\mathrm{f}_n(F).$ Every element in $\textup{Ker}(p)$ is a product of elements of lower weight, so, by the induction hypothesis and \autoref{PropertiesIteratedIntegrals} property \ref{rel:shuffle}, the element $(a-a')\in \mathcal{H}^\mathrm{f}_n(F)$ can be expressed as a linear combination of iterated integrals. 
\end{proof}

The distribution relations also hold for iterated integrals in the Hopf algebra $\mathcal{H}^\mathrm{f}_n$.

\begin{lemma} \label{LemmaDistributionInt}
\begin{enumerate}
\item
Let $N\in \mathbb{N}$ be such that $(N,\ch(F))=1$. Assume that $F$ contains all $N\!$-th roots of unity. Then, for all $n\geq1$, we have
\[
\I(x_0^N;x_1^N,\dots,x_n^N;x_{n+1}^N)=\sum_{\zeta_1^N=1,\dots,\zeta_{n}^N=1}\I(x_0;\zeta_{1}x_1,\dots,\zeta_{n}x_n;x_{n+1})\,.
\]
\item Assume that $F$ is a field of characteristic $p$. Then the Frobenius map which sends 
$\I(x_0;x_1,\dots,x_n,x_{n+1})$ to $\I(x_0^p;x_1^p,\dots,x_n^p;x_{n+1}^p)$ acts as multiplication by $p^n$.
\end{enumerate}
\end{lemma}
\begin{proof}
We prove the first statement; the second statement is similar. We argue by induction on $n$; the base case is obvious. It is easy to see that by the induction assumption we have 
\begin{equation} \label{FormulaDistributionCoproduct}
\Delta'\bigg(\I(x_0^N;x_1^N,\dots,x_n^N;x_{n+1}^N)-\sum_{\zeta_1^N=1,\dots,\zeta_{n}^N=1}\I(x_0;\zeta_{1}x_1,\dots,\zeta_{n}x_n;x_{n+1})\bigg)=0\,,
\end{equation}
where \( \Delta' = \Delta - 1 \otimes \id - \id \otimes 1 \) is the reduced coproduct.
On the other hand, by \autoref{LemmaDistributionCor}, 
\begin{equation} \label{FormulaDistributionCoproductp}
p\bigg(\I(x_0^N;x_1^N,\dots,x_n^N;x_{n+1}^N)-\sum_{\zeta_1^N=1,\dots,\zeta_{n}^N=1}\I(x_0;\zeta_{1}x_1,\dots,\zeta_{n}x_n;x_{n+1})\bigg)=0\,.
\end{equation}
The statement then follows from the fact that $p$ induces an isomorphism between the set of primitive elements in $\mathcal{H}^\mathrm{f}(F)$ and $\ker(\delta)\subseteq \mathcal{L}^\mathrm{f}(F)$, see \cite{MM65}.
\end{proof}

\section{Multiple polylogarithms}\label{sec:MultiplePolylogarithms}
\subsection{Multiple polylogarithms and the depth filtration}\label{SectionMultiplePolylogarithms} For an integer $n_0\geq 0$, $k\ge1$, positive integers $n_1,\dots,n_k$, and elements $a_1,\dots,a_k\in F^\times$, we define the
multiple polylogarithm by
\begin{equation}\label{EquationLeibnizGeneral}
\begin{aligned}
	&\Li_{n_0\semi{}n_1,\dots,n_k}(a_1,a_2,\dots,a_k)\\
& \coloneqq (-1)^{k}\I(0;\underbrace{0,\dots,0,1}_{n_0+1},\underbrace{0,\dots,0,a_1}_{n_1},\dots,\underbrace{0,\dots,0,a_1a_2\dots a_{k-1}}_{n_{k-1}},\underbrace{0,\dots,0;a_1a_2\dots a_{k}}_{n_k})\,.
\end{aligned}
\end{equation}
This is motivated by \eqref{EquationLeibnizKontsevich}, but we also introduce an additional parameter \( n_0 \) for convenience when switching between iterated integrals and multiple polylogarithms.
This is an element of $\mathcal{H}^\mathrm{f}_n(F)$ for $n=n_0+n_1+n_2+\dots+n_k$. 
If $n_0=0$, we omit it from the notation: 
\[
\Li_{n_1,\dots,n_k}(a_1,a_2,\dots,a_k) \coloneqq \Li_{0\semi{}n_1,\dots,n_k}(a_1,a_2,\dots,a_k)\,.
\]
We will denote projection of $\Li_{n_0\semi{}n_1,\dots,n_k}(a_1,\dots,a_k)$ to $\mathcal{L}^\mathrm{f}(F)$ by $\Li_{n_0\semi{}n_1,\dots,n_k}^{\mathcal{L}}(a_1,\dots,a_k)$.

We say that the multiple polylogarithm $\Li_{n_0\semi{}n_1,\dots,n_k}(a_1,a_2,\dots,a_k)$ has depth $k$. We define a filtration on $\mathcal{H}^\mathrm{f}(F)$ by letting $\mathcal{D}_k\mathcal{H}^\mathrm{f}(F)$ be the subspace of $\mathcal{H}^\mathrm{f}(F)$ spanned by polylogarithms of depth at most $k.$ Below we prove that  $\mathcal{D}_{k_1} \cdot \mathcal{D}_{k_2}\subseteq \mathcal{D}_{k_1+k_2},$ so $\mathcal{H}^\mathrm{f}(F)$ is a filtered algebra. We will use the same notation for the induced filtration on $\mathcal{L}^\mathrm{f}(F)$. 

The depth filtration can also be described in terms of iterated integrals and correlators.
\begin{proposition} \label{DepthIIandCor}
The subspace $\mathcal{D}_k\mathcal{H}^\mathrm{f}_n(F)$ is spanned by $\I(0;x_1,\dots,x_n;x_{n+1})$ as $(x_1,\dots,x_{n+1})$ runs over $(n+1)$-tuples of elements in $F$ with at most $k$ non-zero entries among $x_1,\dots,x_n$. The space $\mathcal{D}_k\mathcal{L}^\mathrm{f}_n(F)$ is spanned by $\cor(x_0,\dots,x_n)$ as $(x_0,\dots,x_{n})$ runs over $(n+1)$-tuples of elements with at most $k+1$ non-zero entries. Moreover, if $n\ge2$, and there is only one non-zero entry among $x_0,\dots,x_n$, then $\cor(x_0,\dots,x_n)=0$.
\end{proposition}
\begin{proof}
The first claim is evident from the definition of $\Li_{n_0\semi{}n_1,\dots,n_k}$ and $\mathcal{D}_k$, the only nontrivial part of the claim is that $\I(0; 0,\dots, 0; x_{n+1})\in \mathcal{D}_1\mathcal{H}^\mathrm{f}_n(F)$, which follows from 
\begin{equation} \label{eq:inversiondepth1}
\Li_{n}(x_{n+1})+(-1)^n\Li_n(1/x_{n+1}) = -\I(0; 0,\dots, 0; x_{n+1}).
\end{equation} 
Note that $\I(0; 0,\dots, 0; x_{n+1})=\frac{(\log^{\mathcal{H}}(x_{n+1}))^n}{n!}$, where we denote by $\log^{\mathcal{H}}$.
To prove~\eqref{eq:inversiondepth1}, we denote by $R$ the difference between the LHS and the RHS of~\eqref{eq:inversiondepth1} and note that projection of $R$ to $\mathcal{L}_n^\mathrm{f}(F)$ vanishes by affine invariance of correlators and~\eqref{eqn:correv}. This proves~\eqref{eq:inversiondepth1} for $n=1$, and for general $n$ using 
\begin{equation} \label{eq:lincoproduct}
\Delta'\Li_n(x) = \sum_{j=1}^{n-1}\Li_{n-j}(x)\otimes \frac{(\log^{\mathcal{H}}(x))^j}{j!}
\end{equation} 
and $\Delta' \frac{(\log^{\mathcal{H}}(x))^n}{n!} = \sum_{j=1}^{n-1}\frac{(\log^{\mathcal{H}}(x))^{n-j}}{(n-j)!}\otimes \frac{(\log^{\mathcal{H}}(x))^j}{j!}$, it follows by induction that $\Delta'R=0$, and since $p(R)=0$, we get~\eqref{eq:inversiondepth1} for all $n\ge1$.

For the second claim, \autoref{prop:corasint} implies that if there are at most $k+1$ non-zero entries among $x_0,\dots,x_{n+1}$ then $\cor(x_0,\dots,x_n)\in \mathcal{D}_k\mathcal{L}^\mathrm{f}_n(F)$, since any term $\IL(0; 0,\dots, 0 , x_0, \ldots, x_{n-i-1}; x_{n-i})$ for $i\ge0$ either vanishes (if $x_{n-i}=0$) or lies in $\mathcal{D}_k\mathcal{L}^\mathrm{f}_n(F)$ by the first claim (if $x_{n-i}\ne 0$). To prove that these elements span $\mathcal{D}_k\mathcal{L}^\mathrm{f}_n(F)$, note that $\IL(0;x_1,\dots,x_n;x_{n+1})$ with at most $k$ non-zero entries among $x_1,\dots,x_n$ is equal to $\cor(x_1,x_2,\ldots,x_{n+1}) - \cor(0, x_1,\ldots, x_n)$ with both correlators in $\mathcal{D}_k\mathcal{L}^\mathrm{f}_n(F)$. The final remark follows from properties \ref{Ascale} and \ref{Aone} for correlators.
\end{proof}

\begin{corollary} The depth filtration makes $\mathcal{H}^\mathrm{f}(F)$ a filtered algebra.
\end{corollary}
\begin{proof}
By \autoref{DepthIIandCor}, the space $\mathcal{D}_k\mathcal{H}^\mathrm{f}_n(F)$ is spanned by $\I(0;x_1,\dots,x_n;x_{n+1})$ with at most $k$ non-zero entries among $x_1,\dots,x_n$. Then the shuffle product~\ref{rel:shuffle} immediately implies that $\mathcal{D}_{k_1}\mathcal{H}^\mathrm{f}(F)\cdot \mathcal{D}_{k_2}\mathcal{H}^\mathrm{f}(F)\subseteq \mathcal{D}_{k_1+k_2}\mathcal{H}^\mathrm{f}(F)$.
\end{proof}

\begin{remark} The filtered vector space $\mathcal{H}^\mathrm{f}(F)$ is not a filtered coalgebra. If it were, the reduced coproduct of depth 1 elements would be 0, but from~\eqref{eq:lincoproduct} we see that in \( \mathcal{H}^\mathrm{f}(F) \) already \( \Delta' \Li_2(x) = \Li_1(x) \otimes \log^{\mathcal{H}}(x) \neq 0 \). Similarly, the filtered vector space $\mathcal{L}^\mathrm{f}(F)$ is not a filtered Lie coalgebra.
\end{remark}

\begin{corollary} \label{cor:depthwithoutn0}
The space $\mathcal{D}_k\mathcal{L}^\mathrm{f}_{n}(F)$ is spanned by $\LiL_{n_1,\dots,n_{\ell}}(a_1,a_2,\dots,a_{\ell})$ as $(a_1,\dots,a_{\ell})$ runs over $\ell$-tuples in $(F^{\times})^{\ell}$, $\ell\le k$ and $n_1+\dots+n_{\ell}=n$.
\end{corollary}
\begin{proof}
Fix \( (x_1,\ldots, x_{n+d}) = (0,\ldots,0, y_1,\ldots,y_d) \), with \( y_1 \neq 0 \).  Then in the formula from 
\autoref{PropShuffleIter},
\[\sum_{\sigma \in \Sigma_{n,d}} \IL(0;x_{\sigma(1)},\dots,x_{\sigma(n+d)}; b) = 0\,,\]
there is a unique term \( I(0; 0,\ldots,0, y_1,\ldots,y_d; b) \) with $n$ leading 0's.  Hence one can recursively express it via terms with strictly fewer leading 0's.  Since every term has the same number of non-zero entries (i.e. depth as a multiple polylogarithm), this shows that $\mathcal{D}_k\mathcal{L}^\mathrm{f}_{n}(F)$ is spanned by $\IL(0;x_{1},\dots,x_{n}; x_{n+1})$ with $x_1\ne0$ and at most $k$ non-zero entries among $x_1,\dots,x_n$, which implies the claim.
\end{proof}

\begin{proposition} Let $F$ be an infinite field and $s_0\in F$. The specialization map $\Sp_{s\to s_0}\colon \mathcal{L}^\mathrm{f}(F(s))\lra\mathcal{L}^\mathrm{f}(F)$ is a map of filtered vector spaces. Moreover, for $f_1(s),\dots,f_n(s)\in F(s)^{\times}$ such that $\nu_{s_0}(f_i)$ is not equal to zero for some $1\leq i\leq n$, we have 
\[
\Sp_{s \to s_0} \LiL_{n_0\semi{}n_1,\ldots,n_k}(f_1(s),\ldots,f_k(s))\in \mathcal{D}_{k-1}\mathcal{L}^\mathrm{f}_n(F)\,.
\]
\end{proposition} 
	
\begin{proof}
The first claim is clear from the above remarks on depth of iterated integrals,
since specialization does not increase the number of non-zero entries. 
For the second claim, by \autoref{DepthIIandCor}
    \begin{equation} \label{eq:polylogcorrelator}
    \begin{aligned}[c]
    & \LiL_{n_0;n_1,\dots,n_k}(a_1,a_2,\dots,a_k)\\
    & -(-1)^{k}\cor(\underbrace{0,\dots,0,1}_{n_0+1},\underbrace{0,\dots,0,a_1}_{n_1},\dots,\underbrace{0,\dots,0,a_1a_2\dots a_{k}}_{n_k}) \in \mathcal{D}_{k-1}\mathcal{L}^\mathrm{f}(F)\,,
    \end{aligned}
    \end{equation}
and so it would suffice to show that
    \begin{equation} \label{eq:licordepth}
    \Sp_{s \to s_0}\cor(\underbrace{0,\dots,0,1}_{n_0+1},\underbrace{0,\dots,0,f_1(s)}_{n_1},\dots,\underbrace{0,\dots,0,f_1(s)f_2(s)\dots f_{k}(s)}_{n_k}) \in \mathcal{D}_{k-1}\L_{n}(F)\,.
    \end{equation}
Here as usual $n=n_0+\dots+n_k$.
If $n=k$, then more generally for any $x_0,\dots,x_n$ we have $\cor(x_0,\dots,x_n)=\cor(0,x_1-x_0,\dots,x_n-x_0)\in\mathcal{D}_{k-1}\L_{n}(F)$, and there is nothing to prove. Otherwise, $n>k$, and at least one of the arguments of $\cor$ in~\eqref{eq:licordepth} is~$0$, so to compute specialization we can use \autoref{rem:spec0}. But by the conditions of the proposition, we have $\nu_{s_0}(f_1\dots f_i)\ne \nu_{s_0}(f_1\dots f_{i-1})$, so after specialization one of the non-zero elements will vanish, and the resulting correlator will have $\le k$ non-zero entries, hence we get an element of $\mathcal{D}_{k-1}\L_{n}(F)$ as claimed.
\end{proof}

\begin{remark}
In the special case where exactly one \( \nu_{s_0}(f_j) \) is non-zero for some \( 1 \leq j \leq n \), i.e., when exactly the one function \( f_j(s) \) specializes to \( 0 \) at \( s = s_0 \) and all remaining functions are finite and non-zero at \( s = s_0 \), one obtains more precisely that
\[
    \Sp_{s\to s_0} \Li_{n_0 \semi{} n_1,\ldots,n_k}(f_1(s), \ldots, f_k(s)) = 0 \,.
\]
\end{remark}

\begin{proposition} \label{cor:int:dihedral}
Let $n\ge2$, let $k$ be the number of non-zero entries among $x_1,\dots,x_{n}$, and let $1\le i\le n$ be such that $x_ix_{n+1}\ne0$. Then
 \begin{equation} \label{eq:iterdihedral}
    \IL(0; x_1,\ldots,x_n; x_{n+1}) 
   -\IL(0; x_{i+1}, \ldots, x_{n+1}, x_1, \ldots, x_{i-1}; x_{i})\in \mathcal{D}_{k-1}\mathcal{L}^\mathrm{f}(F)\,.
 \end{equation}
\end{proposition}
\begin{proof}
Writing iterated integrals in terms of correlators and using \ref{Acycle} we get that~\eqref{eq:iterdihedral} equals
\begin{equation*}
    \cor(0,x_{i+1}, \ldots, x_{n+1}, x_1, \ldots, x_{i-1}) - \cor(0, x_1,\ldots,x_n)\,.
\end{equation*}
Since $x_i\ne 0$, the number of non-zero entries in each correlator above is at most $k$, and the claim follows by~\autoref{DepthIIandCor}.
\end{proof}

\subsection{The quasi-shuffle relation}  \label{sec:propertiesofmp}
Our goal is to formulate a version of the quasi-shuffle relation, an identity for multiple polylogarithms which is easy to see for the power series presentation \eqref{FormulaPolylogarithm} but is rather nontrivial in our setting, see \cite{HoffmanQuasi00}, \cite[\S2.5]{Gon01}, \cite[Theorem 1.2]{GonPeriods02}, \cite{Mal20}. To state the quasi-shuffle relation, recall the construction of  quasi-shuffle algebras \cite{HoffmanQuasi00,HoffmanIharaQuasi17}. 

Consider the alphabet \( \mathcal{A} = \{ (n_i, x_i) \mid n_i \in \mathbb{Z}_{>0}, x_i \in F^\times  \} \).  There is a product on \( \mathcal{A} \) defined by \( (n,x) \cdot (m,y) = (n+m, x y) \).  Let \( \mathbb{Q}\langle \mathcal{A} \rangle \) be the vector space of \( \mathbb{Q} \)-linear combinations of words (non-commutative monomials) over \( \mathcal{A}\), which we write as \( [n_1,x_1\mid n_2,x_2\mid \cdots\mid  n_k,x_k] \) for notational ease.  Finally define the quasi-shuffle product \( \star \) on \( \mathbb{Q}\langle \mathcal{A} \rangle \) recursively by
	\begin{equation}\label{eqn:quasishuffle:def}
	\left\{ \begin{aligned}
			\mathbbm{1} \star \omega  &= \omega \star \mathbbm{1} = \omega \,, \\
			 \quad a \omega \star b\eta& = (a\cdot b) \big(\omega \star \eta\big) + b \big( a \omega \star \eta \big) + a \big( \omega \star b \eta \big) \,,
	\end{aligned} \right.
	\end{equation}
	where \( a, b \in \mathcal{A} \) are letters, \( \omega, \eta \in \mathbb{Q}\langle\mathcal{A}\rangle \) are words, and \( \mathbbm{1} \) denotes the empty word.  This gives \( (\mathbb{Q}\langle\mathcal{A}\rangle, \star) \) the structure of a quasi-shuffle algebra.  
	
Define a $\Q$-linear map $\Li \colon (\mathbb{Q}\langle \mathcal{A} \rangle, \star)\lra  \mathcal{H}^\mathrm{f}(F)$ by the formula
\begin{equation}
	\Li([n_1,x_1 \mid  n_2, x_2 \mid  \cdots \mid  n_k,x_k]) = \Li_{n_1,n_2,\ldots,n_k}(x_1,x_2,\ldots,x_k) \,.
\end{equation}

\begin{proposition} The map $\Li$ is a homomorphism of algebras.
\end{proposition}
\begin{proof} The proof of \cite[Proposition 3.9]{Rud20} works in our setting. 
\end{proof}

\begin{remark}\label{rk:generalquasishuffle}
A version of the quasi-shuffle relation which holds for \( \Li_{n_0\semi{}n_1,\ldots,n_k} \), with \( n_0>0 \), was proven in \cite[Proposition 3.10]{Rud20}. To state it, consider a $\Q$-linear map 
$\LiL_{\bullet} \colon \mathbb{Q}\langle \mathcal{A}\rangle \longrightarrow \mathcal{L}^\mathrm{f}(F)$ defined by the formula
\begin{equation}\label{eqn:Libullet}
	\LiL_{\bullet} [n_1,x_1 \mid  n_2, x_2 \mid  \cdots \mid  n_k,x_k]=\sum_{n_0\geq 0} \LiL_{n_0\semi{}n_1,n_2,\ldots,n_k}(x_1,x_2,\ldots,x_k) \,.
\end{equation}
Then quasi-shuffle relation $\LiL_\bullet(\omega \star \eta) = 0$ holds for  \( \omega, \eta \neq \mathbbm{1} \in \mathbb{Q}\langle \mathcal{A} \rangle \).
\end{remark}

	Using the quasi-shuffle structure, and the dihedral symmetries of correlators (resp. iterated integrals, modulo lower depth), one can now establish a number of useful general identities in $\mathcal{L}^\mathrm{f}(F)$. 

	\begin{lemma}\label{lem:stuffle:antipode}
		The following symmetry holds
		\begin{equation}\label{eqn:stuffleantipode}
			\LiL_{n_0\semi{}n_1,\ldots,n_k}(x_1,\ldots,x_k) + (-1)^{k} \LiL_{n_0\semi{}n_k,\ldots,n_1}(x_k,\ldots,x_1) \in \mathcal{D}_{k-1}\mathcal{L}^\mathrm{f}_n(F) \,.
		\end{equation}
		
		\begin{proof}
            Recall that shuffle product \( \shuffle \), which is defined on the alphabet \( \mathcal{A} \), is given by
				\begin{equation}\label{eqn:shuffledef}
				\left\{ \begin{aligned}
			\mathbbm{1} \shuffle \omega  &= \omega \shuffle \mathbbm{1} = \omega \,, \\
			\quad a \omega \shuffle b\eta& = b \big( a \omega \shuffle \eta \big) + a \big( \omega \shuffle b \eta \big) \,.
			\end{aligned} \right.
			\end{equation}
			Given two words \( \eta, \omega \in \Q\langle A \rangle \) with total length \( k \), then
            \[
                \eta \star \omega = \eta \shuffle \omega + \text{(words of length $<k$)} \,,
            \]
            as we are neglecting the terms which come from the factor \( (a\cdot b) \big(\omega \star \eta\big) \) in \eqref{eqn:quasishuffle:def}, which are necessarily of length \( <k \).   The following identity holds, by induction, in any shuffle algebra (see \cite[Eqn. (29)]{Gon05}, \cite[Thm~4.1]{Gon01B})
			\[
				\sum_{i=0}^k (-1)^i \cdot (a_1 a_2 \cdots a_i \shuffle a_k \cdots a_{i+2} a_{i+1}) = 0 \,.
			\]
            Specializing to letters \( a_i = (n_i, x_i)  \in \mathcal{A} \), we obtain the following, for the quasi-shuffle sum
			\[
				\sum_{i=0}^k (-1)^i (a_1 a_2 \cdots a_i) \star (a_k \cdots a_{i+2} a_{i+1}) = \text{(words of length $<k$)}\,.
			\]
            Now apply the $\Q$-linear map \( \LiL_\bullet \) from \eqref{eqn:Libullet}.  The RHS is in \( \mathcal{D}_{k-1} \mathcal{L}^\mathrm{f}(F) \), while for \( 1 \leq i \leq k-1 \) the summand on the LHS vanishes by \autoref{rk:generalquasishuffle} as it consists of two non-empty words.  Only the \( i = 0, k \) summands survive, giving
			\[
				(-1)^k \LiL_\bullet( a_1 a_2 \cdots a_k) + \LiL_\bullet(a_k \cdots a_{2} a_{1}) \in \mathcal{D}_{k-1} \mathcal{L}^\mathrm{f}(F) \,,
			\]
			Extracting the weight-graded pieces of \( \LiL_\bullet \) gives us the desired identity.
		\end{proof}
	\end{lemma}

	\begin{remark}
		One can more precisely describe the lower depth terms, using the ``star-version'' of the multiple polylogarithms (cf. the multiple zeta star values and the  interpolated multiple zeta values \cite{YamamotoInterpolation13}).  See \S4.2.1.2 and Lemma 4.2.2 in \cite{GlanoisThesis16} or Lemma 3.3 in \cite{GlanoisBasis16}, in particular.  Alternatively \cite[\S3]{HoffmanQuasi20}, and the proof of Theorem 1.3 \cite{HofChMtVSym22} give a result stated in terms of a general interpolated product.
	\end{remark}

	\begin{lemma}[Inversion, {cf. Goncharov \cite[\S2.6]{Gon01}, Panzer \cite{PanzerParity} for an analytic version}] \label{LemmaInversionGeneral}
		In weight \( n = n_0 + n_1 + \cdots + n_k \geq 2 \), the following symmetry holds,
		\[
			\LiL_{n_0\semi{}n_1,\ldots,n_k}(x_1,\ldots,x_k) - (-1)^{n + k} \LiL_{n_0 \semi{} n_1,\ldots,n_k}(x_1^{-1}, \ldots, x_k^{-1}) \in \mathcal{D}_{k-1} \mathcal{L}^\mathrm{f}_n(F) \,,
		\]
		
		\begin{proof}
			This follows by the reversal symmetry of integrals \eqref{eqn:revsym} (in \autoref{PropDihedralIter}), the cyclic symmetry of integrals \eqref{eq:iterdihedral}  (in \autoref{cor:int:dihedral}), the reversal symmetry \eqref{eqn:stuffleantipode} (in \autoref{lem:stuffle:antipode}) of multiple polylogarithms, and the affine invariance of integrals \eqref{eqn:int:affine} (in \autoref{cor:int:affine}).  Write \( X_{i,j} = \prod_{\ell=i}^j x_\ell \) for brevity (with \( X_{1,0} = 1 \) as an empty product), and introduce the shorthand \( \{0\}^a = 0,\ldots,0 \) ($a$ repetitions).  Then
			{
			\begin{align*}
				& \LiL_{n_0\semi{}n_1,\ldots,n_k}(x_1,\ldots,x_k) \\
				& = \begin{aligned}[t] (-1)^k \IL(0; \{0\}^{n_0}, X_{1,0}, \{0\}^{n_1\-1}, X_{1,1}, \{0\}^{n_2\-1}, \ldots, X_{1,k\-1}, \{0\}^{n_k\-1}; X_{1,k} ) \end{aligned} \\[1ex]
			\tag{Eqn. \eqref{eqn:revsym}}	& = \begin{aligned}[t] (-1)^{k+n+1} \IL(0; \{0\}&^{n_k\-1}, X_{1,k\-1}, \{0\}^{n_{k\-1}\-1} , X_{1,k\-2}, \ldots, \\[-0.5ex] 
			& \{0\}^{n_2\-1}, X_{1,1}, \{0\}^{n_1\-1}, X_{1,0}, \{0\}^{n_0} ; X_{1,k} ) \end{aligned} \\[1ex]
			\tag{Eqn. \eqref{eqn:int:affine}}	& = \begin{aligned}[t] (-1)^{k+n+1} \IL(0; \{0\}&^{n_k\-1}, \tfrac{X_{1,k\-1}}{X_{1,k\-1}}, \{0\}^{n_{k\-1}\-1} ,   \tfrac{X_{1,k\-2}}{X_{1,k\-1}}, \ldots, \\[-0.5ex] & \{0\}^{n_2\-1}, \tfrac{X_{1,1}}{X_{1,k\-1}}, \{0\}^{n_1\-1}, \tfrac{X_{1,0}}{X_{1,k\-1}}, \{0\}^{n_0} ; \tfrac{X_{1,k}}{X_{1,k\-1}} ) \end{aligned} \\[1ex]
				& = (-1)^{n+1} \LiL_{n_k\-1 \semi{} n_{k\-1}, \ldots, n_1, n_0\+1}\big( \tfrac{X_{1,k\-2}}{X_{1,k\-1}} \mkern-2mu\big/\mkern-2mu \tfrac{X_{1,k\-1}}{X_{1,k\-1}}, \ldots,  \tfrac{X_{1,0}}{X_{1,k\-1}} \mkern-2mu\big/\mkern-2mu \tfrac{X_{1,1}}{X_{1,k\-1}} , \tfrac{X_{1,k}}{X_{1,k\-1}} \mkern-2mu\big/\mkern-2mu \tfrac{X_{1,0}}{X_{1,k\-1}} \big) \\
				& = (-1)^{n+1} \LiL_{n_k\-1 \semi{} n_{k\-1}, \ldots, n_1, n_0\+1}\big( x_{k-1}^{-1}, \ldots, x_1^{-1} , X_{1,k} \big) \\[1ex]
			\tag{Eqn. \eqref{eqn:stuffleantipode}}	\hspace{1em} & \equiv (-1)^{k + n} \LiL_{n_k\-1 \semi{} n_0\+1, n_1, \ldots, n_{k\-1}}\big( X_{1,k}, x_1^{-1}, \ldots, x_{k-1}^{-1} \big) \pmod{\mathcal{D}_{k-1}\mathcal{L}^\mathrm{f}_n(F)} \,.
			\end{align*}
			}
			Now we rewrite the last expression again as an iterated integral \( \IL \) using \eqref{EquationLeibnizGeneral}, giving
			{
			\begin{align*}
				& = (-1)^{n} \IL(0; \{0\}^{n_k\-1}, 1, \{0\}^{n_0}, \tfrac{X_{1,k}}{X_{1,0}}, \{0\}^{n_1\-1}, \tfrac{X_{1,k}}{X_{1,1}}, \ldots, \tfrac{X_{1,k}}{X_{1,k-2}}, \{0\}^{n_{k\-1}-1} ; \tfrac{X_{1,k}}{X_{1,k-1}} )  \\
				\tag{Eqn. \eqref{eq:iterdihedral}}
				& \begin{aligned} \equiv (-1)^{n} \IL(0; \{0\}^{n_0}, \tfrac{X_{1,k}}{X_{1,0}}, \{0\}^{n_1\-1}, \tfrac{X_{1,k}}{X_{1,1}}, \ldots, \tfrac{X_{1,k}}{X_{1,k-2}}, \{0\}^{n_{k\-1}-1} , \tfrac{X_{1,k}}{X_{1,k-1}}, \{0\}^{n_k\-1}; 1 ) \qquad \\ \pmod{\mathcal{D}_{k-1}\mathcal{L}^\mathrm{f}_n(F)}  \end{aligned}\\[1ex]
				\tag{Eqn. \eqref{eqn:int:affine}} 
				& = (-1)^{n} \IL(0; \{0\}^{n_0}, \tfrac{1}{X_{1,0}}, \{0\}^{n_1\-1}, \tfrac{1}{X_{1,1}}, \ldots, \tfrac{1}{X_{1,k-2}}, \{0\}^{n_{k\-1}-1} , \tfrac{1}{X_{1,k-1}}, \{0\}^{n_k\-1}; \tfrac{1}{X_{1,k}} )  \\
				& = (-1)^{n+k} \LiL_{n_0\semi{}n_1,\ldots,n_k}(x_1^{-1}, \ldots, x_k^{-1}) \,.
			\end{align*}
   }
			 This is equivalent to the result we wanted to show, and so the lemma is proven.
		\end{proof}
	\end{lemma}
 
\subsection{Goncharov Conjectures}\label{sec:gonconj}
Goncharov suggested several conjectures about the structure of the Lie coalgebra of mixed Tate motives $\L^{\M}(F).$ In this section, we state the analogues of these conjectures for $\mathcal{L}^\mathrm{f}(F).$ We start with the connection with $K$-theory.

\begin{conjecture} \label{ConjecturePolylogsKtheory} Let $F$ be an infinite field. For $1\leq i \leq n$ there  exist isomorphisms
\begin{equation}\label{FormulaPolylogsKtheoryLie}
\textup{ch}_{n,i}\colon \gr_{\gamma}^{n} K_{2n-i}(F)_{\Q} \lra  H^{i}(\mathcal{H}^\mathrm{f}(F),\Q)_n.
\end{equation}
\end{conjecture}

\begin{remark}
The maps $\textup{ch}_{n,i}$ should be proportional to the Chern class maps from algebraic $K$-theory to motivic cohomology. 
\end{remark}

Our next goal is to formulate Goncharov's Freeness and Depth conjectures.  Our definition of the Lie coalgebra $\mathcal{L}^\mathrm{f}(F)$ is inspired by Goncharov's definition of higher Bloch groups $\mathcal{B}_n(F)$ defined in \cite{Goncharov94} and \cite{Gon95B} (where the same object is called $\mathcal{B}_n'(F))$. It is easy to see that there exists a map $i\colon \mathcal{B}_n(F)\lra \mathcal{L}^\mathrm{f}_n(F)$ sending $\{a\}_n\in \mathcal{B}_n(F)$ to $\LiL_n(a)\in \mathcal{L}^\mathrm{f}_n(F);$ the image of $i$ is the subspace $\mathcal{D}_1\mathcal{L}^\mathrm{f}_n(F)$ spanned by classical polylogarithms. We expect that $i$ induces an isomorphism $\mathcal{B}_n(F)\xrightarrow{\cong?} \mathcal{D}_1\mathcal{L}^\mathrm{f}_n(F)$ but do not see how to prove that $i$ is injective.

Our next goal is to formulate the Depth Conjecture, which gives a conjectural motivic description of the depth filtration on $\mathcal{L}^\mathrm{f}(F).$ Assume that \( \delta = \sum_{1 \leq i \leq j} \delta_{ij} \), where \( \delta_{ij} \colon \L_{i+j}(F) \to \L_i(F) \wedge \L_j(F) \). The truncated cobracket is the map
\begin{equation*}
	\overline{\delta} \colon \mathcal{L}^\mathrm{f}(F) \to \bigwedge\nolimits^2 \mathcal{L}^\mathrm{f}(F) 
\end{equation*}
defined by the formula \( \overline{\delta} = \sum_{2 \leq i \leq j} \delta_{ij} \), i.e. \( \overline{\delta} \) is obtained from \( \delta \) by omitting the \( \L_1(F) \wedge \L_{n-1}(F) \) component from the cobracket. The truncated cobracket can be restricted to the quotient $\overline{\mathcal{L}^\mathrm{f}}(F):=\mathcal{L}^\mathrm{f}(F)/\mathcal{L}^\mathrm{f}_1(F)$ giving it a structure of a Lie coalgebra, which carries the \emph{depth filtration} induced from $\mathcal{L}^\mathrm{f}(F).$

\begin{lemma} Let $F$ be an infinite field. For $k\geq 1$ we have 
\[
\overline{\delta}(\mathcal{D}_k\overline{\mathcal{L}^\mathrm{f}}(F))\subseteq \sum_{k_1+k_2=k} \mathcal{D}_{k_1}\overline{\mathcal{L}^\mathrm{f}}(F) \wedge  \mathcal{D}_{k_2}\overline{\mathcal{L}^\mathrm{f}}(F) 
\]
In other words, the depth filtration gives $\bigl(\overline{\mathcal{L}^\mathrm{f}}(F),\bar{\delta}\bigr)$ the structure of a filtered Lie coalgebra.
\end{lemma}
\begin{proof}
In terms of correlators $\mathcal{D}_k\overline{\mathcal{L}^\mathrm{f}_{n}}(F)$, $n\ge2$ is spanned by $\cor(x_0,\dots,x_n)$ with at most $k+1$ non-zero entries among $x_i$. Evidently, in
    \[\bar{\delta} \cor(x_0,\dots, x_n)=\sum_{j=0}^n\sum_{i=2}^{n-2} \cor(x_{j}, x_{j+1}, \dots, x_{j+i})\wedge  \cor(x_{j},  x_{j+i+1}, \dots, x_{j+n})\]
the total number of non-zero $x$'s in any term on the right-hand side is $(k_1+1)+(k_2+1)\le k+2$ (since the two tuples $(x_j,x_{j+1},\dots,x_{j+i})$ and $(x_j,x_{j+i+1},\dots,x_{j+n})$ overlap only in $x_j$), and so by the second part of~\autoref{DepthIIandCor} we get the claim (since $n\ge2$, when $k_1=0$ or $k_2=0$, the corresponding term in the coproduct vanishes).
\end{proof}

From this lemma it follows that the truncated cobracket $\overline{\delta}$ vanishes on classical polylogarithms $\LiL_n(a)\in \mathcal{D}_1\mathcal{L}^\mathrm{f}(F).$ Goncharov conjectured that the converse is true; we formulate a version of his conjecture in our setting.

\begin{conjecture}\label{conj:classicalpolylog} For $n\geq 1$ an element $x\in \mathcal{L}^\mathrm{f}_{n}(F)$ is a linear combination of multiple polylogarithms $\Li_n(a)$ for $a\in F^{\times}$ if and only if $\overline{\delta}(x)=0.$
\end{conjecture}

\begin{remark}
The proof of \cite[Corollary 6]{CGRR22} can be adapted to show that, for a quadratically closed field $F$, \autoref{conj:classicalpolylog} would imply that the Lie coalgebra  $\bigl(\overline{\mathcal{L}^\mathrm{f}}(F),\bar{\delta}\bigr)$ is cofree and that the depth filtration on $\overline{\mathcal{L}^\mathrm{f}}(F)$ coincides with the coradical filtration. These two statements are known as the Goncharov Freeness Conjecture and the Goncharov Depth Conjecture.
\end{remark}

For a vector space $V$ we denote by $\CoLie(V)$ the cofree conilpotent Lie coalgebra on $V$. Recall that the coLie cooperad $Lie^c$ is the graded dual of the Lie operad $Lie$. For $n\geq 0$ we have an isomorphism 
\[
\CoLie_k(V)\cong Lie^c_{k} \otimes_{\Sigma_{k}} V^{\otimes k}.
\]
For any Lie coalgebra $\mathcal{L}$ we have the structure maps 
$
\delta^{[k]}\colon \mathcal{L} \lra \CoLie_{k+1}(\mathcal{L} )$  called {\it iterated cobrackets}.  A more explicit corollary of the Depth Conjecture is the following generalization of \autoref{conj:classicalpolylog}:

\begin{conjecture}\label{conj:depth} For $n\geq 1$ an element $a\in \overline{\mathcal{L}_{n}^\mathrm{f}}(F)$ is a linear combination of multiple polylogarithms of depth at most $k$  if and only if $\overline{\delta}^{[k]}(a)=0,$ where $\overline{\delta}^{[k]}$ is the iterated truncated cobracket
\[
\overline{\delta}^{[k]}\colon \overline{\mathcal{L}^\mathrm{f}}(F)\lra \CoLie_{k+1}\big( \overline{\mathcal{L}^\mathrm{f}}(F)\big) \,.
\]
Moreover, for any $k\geq 1$ we have an isomorphism 
\[
\overline{\delta}^{[k-1]} \colon \gr_k^\mathcal{D} \overline{\mathcal{L}^\mathrm{f}}(F) \xrightarrow{\cong?} \CoLie_k\Big( \bigoplus_{n\geq2} \mathcal{D}_1\overline{\mathcal{L}^\mathrm{f}_n}(F) \Big) \,.
\]
\end{conjecture}

A few cases of the Depth conjecture are already known, and their proofs can be adapted to our setting.  Here are some references. The case $n=4$ was conjectured by Goncharov and proven by Gangl in \cite{gangl-4}. The case $n=6, k=2$ follows from the results of \cite{MR22} and \cite{Cha24}. Finally, one can show that for $k\geq \lfloor n/2 \rfloor$, $\mathcal{D}_{k} \overline{\mathcal{L}^\mathrm{f}_{n}}(F)=\overline{\mathcal{L}^\mathrm{f}_{n}}(F),$ using the same method as in the proof of \cite[Theorem~1]{Rud20}.

\subsection{Weights 2 and 3} \label{sec:wt2wt3}

We have seen that $\mathcal{L}^\mathrm{f}_1(F)\cong F^{\times}_{\Q}.$ Our next goal is to describe $\L_2(F)$ and give a conjectural description of $\L_3(F).$ There is a hope that in every weight there exists a {\it universal} functional equation for multiple polylogarithms, such that all equations in $\mathcal{R}_n(F)$ would follow from it. In weight $2$, this is the $5$-term relation (see \eqref{Equation5term}, below). In weight $3$, the Goncharov $22$-term relation (see \eqref{Equation22term}, below) is conjectured to be universal. These relations are used to define explicit candidates $B_2(F)$ and $B_3(F)$ of the corresponding higher Bloch groups $\mathcal{B}_2(F)$ and $\mathcal{B}_3(F)$. We recall these constructions below.

The Bloch group $B_2(F)$ is defined as a quotient of the group $\Z[F^{\times}\setminus \{1\}]$ by the subgroup spanned by elements 
\begin{equation}\label{Equation5term}
R_2(a,b)=[a]-[b]+\left [ \frac{b}{a}\right ] - \left [ \frac{1-a^{-1}}{1-b^{-1}}\right ] + \left [ \frac{1-a}{1-b}\right ]
\end{equation}
for $a\neq b \in F^{\times}\setminus \{1\}$. We denote the projection of $[a]$ to $B_2(F)$ by $\{a\}_2$. We have a well-defined differential $\delta \colon B_2(F)\lra  \bigwedge^2  F^{\times }$ sending $\{a\}_2$ to $a\wedge (1-a)$ constructed in \cite[\S 1]{Sus90} and specialization $\textup{Sp}_{\nu} \colon B_2(F)\lra  B_2(\mathrm{k})$ constructed in \cite[\S 5]{Sus90}. Suslin proved \cite[Corollary 5.6]{Sus90} that the kernel of the cobracket
is invariant under pure transcendental extensions of $F$:
\begin{equation}\label{FormulaHomotopy}
\ker\Big(\delta \colon B_2(F)_{\Q}\lra \bigwedge\nolimits^2 F_\Q^\times \Big)\cong \ker\Big(\delta \colon B_2(F(t))_{\Q}\lra \bigwedge\nolimits^2F(t)_\Q^\times \Big).
\end{equation}

One can easily check that \( \Li_2^{\L} \) satisfies the 5-term relation \eqref{Equation5term}, therefore we have a well-defined map $L_2\colon B_2(F) \lra \mathcal{L}^\mathrm{f}_2(F)$ sending $\{a\}_2$ to $\Li_2^{\L}(a)$.
We denote by the same symbol the rationalization $L_2\colon B_2(F)_{\Q} \lra \mathcal{L}^\mathrm{f}_2(F).$

\begin{proposition}\label{PropositionWeightTwo} The map $L_2\colon B_2(F)_{\Q} \lra \mathcal{L}^\mathrm{f}_2(F)$ is an isomorphism. 
\end{proposition}
\begin{proof}
To prove the proposition, we construct a map $M_2$ in the opposite direction. Notice that an element $\llp x_0, x_1, x_2 \rrp\in \mathcal{A}_2(F)$ vanishes unless  $x_0, x_1,x_2$ are distinct. Consider a map $M_2\colon \mathcal{A}_2(F)\lra B_2(F)$ sending 
$\llp x_0, x_1,x_2 \rrp$ with distinct $x_0, x_1,x_2$ to $\{\frac{x_2-x_0}{x_1-x_0}\}_2$. It is easy to see that this map commutes with the specialization and with the cobracket. Consider an element $R\in \mathcal{A}_2(F(t))$ with $\delta(R)=0$. It follows that $\delta(M_2(R))=0$ and so $\delta(M_2(R))$ lies in the kernel $\ker\Big(\delta \colon B_2(F(t))_{\Q}\lra \bigwedge\nolimits^2F(t)_\Q^\times \Big).$ By \eqref{FormulaHomotopy}, $M_2(R)$ lies in the subspace $B_2(F)\subseteq B_2(F(t))$, and so 
\[
M_2(\Sp_{t \to 0}(R)-\Sp_{t \to 1}(R))=\Sp_{t \to 0}M_2(R)-\Sp_{t \to 1}M_2(R)=0.
\]
Thus $R_2(F)$ is annihilated by $M_2$, so $M_2$ induces a map $\mathcal{L}^\mathrm{f}_2(F) \lra B_2(F)$ which we denote by the same symbol.  It is easy to see that maps $L_2$ and $M_2$ are mutually inverse.
\end{proof}

In \cite{Gon95B}, Goncharov defined a group $B_3(F)$  as a quotient of $\Z[\mathbb{P}^1_F]$ by the subgroup spanned by elements $[0], [\infty], [a]-\left[a^{-1}\right], [a]+[1-a]+\left[1 - a^{-1}\right]-[1]$ and the \emph{Goncharov  22-term relation} (note the three-fold symmetry \( a \mapsto b \mapsto c \mapsto a \) in each row)
\begin{equation}\label{Equation22term}
\begin{aligned}
R_3(a,b,c) = {} & [ca-a+1]+[ab-b+1]+[bc-c+1]\\
&+ \left[\frac{ca-a+1}{ca}\right]+\left[\frac{ab-b+1}{ab}\right]+\left[\frac{bc-c+1}{bc}\right]\\
&+ \left[\frac{bc-c+1}{(ca-a+1)b}\right]+\left[\frac{ca-a+1}{(ab-b+1)c}\right]+\left[\frac{ab-b+1}{(bc-c+1)a}\right]\\
&- \left[\frac{ca-a+1}{c}\right]-\left[\frac{ab-b+1}{a}\right]-\left[\frac{bc-c+1}{b}\right]\\
&+ \left[-\frac{(bc-c+1)a}{(ca-a+1)}\right]+\left[-\frac{(ca-a+1)b}{(ab-b+1)}\right]+\left[-\frac{(ab-b+1)c}{(bc-c+1)}\right]\\
&-\left[\frac{(bc-c+1)}{(ca-a+1)bc}\right]-\left[\frac{(ca-a+1)}{(ab-b+1)ca}\right]-\left[\frac{(ab-b+1)}{(bc-c+1)ab}\right]\\
& + [a]+[b]+[c]+[-abc]-3 \, [1]
\end{aligned}
\end{equation}
for $a,b,c \in F$ such that expressions $0/0$ and $\infty/\infty$ do not appear in the right hand side.

\begin{conjecture} \label{ConjectureB_3} For an infinite field $F$, the map 
\[
L_3\colon (B_3(F))_{\Q}\lra \mathcal{L}^\mathrm{f}_3(F)
\]
sending  $[a]$ to $\Li_3^{\L}(a)$ is an isomorphism.
\end{conjecture}

\begin{remark} 
By Milnor-Moore, $H^{i}(\mathcal{H}^\mathrm{f}(F),\Q)_n \cong H^{i}(\mathcal{L}^\mathrm{f}(F),\Q)_n$. 
\Autoref{PropositionWeightTwo} implies that $H^{n}(\mathcal{L}^\mathrm{f}(F),\Q)_n$ is isomorphic to the cokernel of the map 
\[
B_2(F)\otimes \Lambda^{n-2}F^{\times}_{\Q}\lra  \Lambda^{n}F^{\times}_{\Q},
\]
which coincides with the rationalization of the  Milnor $K$-group $K_n^M(F).$
Since $\gr_{\gamma}^{n} K_{n}(F)_{\Q}$ is isomorphic $K_n^M(F)_{\Q}$, \Autoref{ConjecturePolylogsKtheory} is known to hold for $i=n$.
\end{remark}
 
\section{Realizations} 
\label{SectionRealizations}
In this section, we discuss the Hodge and motivic realizations of the Lie coalgebra of formal multiple polylogarithms. 

\subsection{Hodge realization} \label{SectionHodgeRealization}
In \cite{Del71,Del71B}, Deligne defined mixed Hodge structures and showed that they form a neutral Tannakian category $\mathrm{MHS}_{\Q}$. A mixed Hodge--Tate structure is a mixed Hodge structure which is an iterated extension of pure Hodge structures $\Q(-n)$, for $n\in \mathbb{Z}.$ A mixed Hodge--Tate structure can be defined more explicitly as a finite dimensional vector space $H$ over $\Q$ along with an
increasing filtration $W_\bullet$ on $H$ and a decreasing filtration $F^\bullet$ on the complexification $H_{\mathbb{C}}$ such that
\[
W_{2n}H_{\mathbb{C}}=W_{2(n-1)}H_{\mathbb{C}}\oplus (W_{2n}H_{\mathbb{C}}\cap F^n H_{\mathbb{C}})\,.
\]
A framed mixed Hodge--Tate structure is a triple $[H,v,\varphi]$ of a mixed Hodge--Tate structure $H$, a non-zero linear map  $\varphi\colon \gr_0^W H\lra \Q(0)$ and a non-zero linear map $v\colon \Q(-n)\lra  \gr_{2n}^W H.$  Denote by $\mathcal{H}_{n}^{\QHod}$  the $\Q$-vector space spanned by framed mixed Hodge--Tate structures $[H,\varphi,v]$ modulo relations  $[H_1,f^{\vee}(\psi),v]=[H_2,\psi,f(v)]$ for any morphism $f\colon H_1\lra H_2,$ a linear map $v\colon \Q(-n)\lra  \gr_{2n}^W H_1$ and a linear map $\psi \colon \gr_0^W H_2 \lra \Q(0).$ The vector space $\mathcal{H}^{\QHod}=\bigoplus_{n\geq 0}\mathcal{H}_{n}^{\QHod}$ is a graded commutative Hopf algebra; the category of mixed Hodge--Tate structures is equivalent to the category of its graded comodules. We denote $\mathcal{L}^{\QHod}$ the Lie coalgebra of indecomposables of $\mathcal{H}^{\QHod}.$ It is known that $\mathcal{H}_{1}^{\QHod}\cong\mathcal{L}_{1}^{\QHod}\cong \mathbb{C}^{\times}_{\Q}$.

Our next goal is to discuss Hodge correlators: a family of elements in $\mathcal{L}^{\QHod}$ introduced by Goncharov in \cite{Gon19}, see also \cite[\S 4]{GR-zeta4} and \cite[\S 2.4]{Gon19B}. Goncharov usually works in the motivic setting and remarks that one can also work with the Hodge realization. 

We start with a linear algebra construction. For a vector space $V$ over $\Q$ we consider a graded vector space $\textup{CLie}^c(V)$ defined as a quotient of the tensor algebra on $V$ by the cyclic symmetry relations \[
v_0\otimes \dots \otimes v_{m-1} \otimes v_m-v_1\otimes\dots \otimes  v_{m} \otimes v_0
\]
and shuffle relations 
\[
\sum_{\sigma \in \Sigma_{m_1,m_2}}v_0\otimes v_{\sigma(1)} \otimes \dots \otimes v_{\sigma(m)} \,, 
\]
for $m_1+m_2=m$, with $ m_1,m_2\geq 1$.
 The space $\textup{CLie}^c(V)$ is a Lie coalgebra with the cobracket 
\begin{equation*}
\delta( v_0\otimes \dots \otimes  v_m ) =\sum_{j=0}^m\sum_{i=1}^{m-1}( v_{j}\otimes v_{j+1}\otimes \dots\otimes v_{j+i}) \wedge (v_{j}\otimes  v_{j+i+1}\otimes \dots \otimes v_{j+m}) \,,
\end{equation*}
where indices are considered modulo $m+1$.

For distinct points $x_1,\dots,x_m\in \mathbb{C}=\mathbb{P}^1\setminus \{\infty\}$, consider the curve $X=\mathbb{C}\setminus \{x_1,\dots,x_m\}$.  The vector space $H_1(X,\mathbb{Q})$ has a basis consisting of loops around the punctures $x_1,\dots, x_m$; we denote by $X^1,\dots, X^m$ the dual basis of $H^1(X,\mathbb{Q})$. For a tangent vector $v \in \mathrm{T}_{\infty}\mathbb{P}^1$ Goncharov defined a map of Lie coalgebras
\[
\textup{CLie}^c(V)\lra \mathcal{L}^{\QHod}.
\]
We will choose $v=\partial_{1/x}$ for the standard coordinate $x$ on $\mathbb{C}$. For any indices $1\le i_0,\dots,i_n\le m$ the image of the element $X^{i_0}\otimes \dots \otimes X^{i_n}$ is called a $\mathbb{\Q}$-Hodge correlator 
\[
\cor^{\QHod}(x_{i_0},\dots,x_{i_n}) \in \L_n^{\QHod}.
\]
In weight one, we have $\cor^{\QHod}(x_0,x_1)=(x_1-x_0)\in \mathbb{C}^{\times}_{\Q}$.

It is easy to see that the map 
\[
\cor^{\QHod}\colon \mathcal{A}_n(\mathbb{C})\longrightarrow \L_n^{\QHod}\, \text{ \ for\ } n\geq 1\,
\]
sending \( \llp x_0,x_1,\ldots,x_n \rrp \) to $\cor^{\QHod}(x_0,\dots,x_n)$ is a well-defined map of Lie coalgebras. Indeed,  \ref{Acycle} is a consequence of the cyclic symmetry relation in $\textup{CLie}^c(V)$ and \ref{Azero},\ref{Aone} follow from shuffle relations. Next, \ref{Alog} follows from an explicit formula the correlator in weight one. The remaining properties \ref{Atranslate}, \ref{Ascale} can be proven by the standard rigidity argument, see \cite[\S 2]{GR-zeta4}.

\begin{proposition}\label{LemmaHodgeRealization} For any $R\in \mathcal{R}_n(\mathbb{C})$ we have $\cor^{\QHod}(R)=0$. So, we have a well-defined map 
\[
r_{\QHod}\colon \mathcal{L}^\mathrm{f}_n(\mathbb{C})\longrightarrow \L_n^{\QHod}
\]
called the Hodge realization.
\end{proposition}

Before proving the Proposition, we make a few remarks about the behavior of Hodge correlators when we vary the points $x_1,\dots,x_m$. Consider the space
\[
\Conf_{m}(\mathbb{C})=\{(s_1,\dots,s_m)\in\mathbb{C}^m\mid s_i\neq s_j, 1\leq i<j\leq m\}\,.
\]
For any  $1\leq i_0,\dots,i_n\leq m$ we obtain an element 
\begin{equation}\label{FormulaHodgeCorrelator}
\cor^{\QHod}(s_{i_0},\dots,s_{i_n})
\end{equation}
in the Lie coalgebra of framed variations of mixed Hodge--Tate structures on $\Conf_{m}(\mathbb{C})$ (see \cite[\S3.3]{Gon13} for the definition; here $s_i$ are viewed as functions on $\Conf_{m}(\mathbb{C})$).

Next, consider rational functions $f_0,\dots,f_n\in \mathbb{C}(t)$. An element 
\begin{equation}\label{FormulaHodgeCorrelatorFunctions}
\cor^{\QHod}(f_0(t),\dots,f_n(t))
\end{equation}
defines an element in the Lie coalgebra  $\L_{\mathbb{C}\setminus S}^{\QHod}$ of framed variations of  mixed Hodge--Tate structures on $\mathbb{C}\setminus S$ for some finite set $S$. Indeed, \eqref{FormulaHodgeCorrelatorFunctions} is obtained as a pull-back of \eqref{FormulaHodgeCorrelator} from $\Conf_{m}(\mathbb{C})$ where $m$ is the number of distinct functions among $f_0,\dots,f_n$. The open subset $\mathbb{C}\setminus S\subseteq \mathbb{C}$ consists of those points $t\in \mathbb{C}$ at which the cardinality of the set $\{f_0(t),\dots, f_n(t)\}$ equals that of the set  $\{f_0,\dots, f_n\}$.

Specializations of framed variations of mixed Hodge-Tate structures are discussed in \cite[\S 4]{Mal20}. For a point $t_0\in S$ we have the specialization map
\[
\Sp_{t\to t_0}\colon \L_{\mathbb{C}\setminus S}^{\QHod} \longrightarrow \L^{\QHod}\,,
\]
such that $\Sp_{t\to t_0}(\cor^{\QHod}(t,t_0))$ is equal to zero.
The following statement follows from \cite[Theorem 28]{Mal20}: 
\begin{equation}\label{FromulaSpecializationMHCorrelators}
\Sp_{t \to t_0}\bigl(\cor^{\QHod}( f_0(t),\dots,f_n(t) )\bigr)=\cor^{\QHod}\bigl(\Sp_{t \to t_0} \llp f_0(t),\dots,f_n(t)\rrp \bigr).
\end{equation}

\begin{proof}[Proof of \autoref{LemmaHodgeRealization}] 
An element 
\[
\widetilde{R}=\sum_{i=1}^N n_i \llp f_0^i(t),\dots,f_n^i(t) \rrp \in \mathcal{A}_n(\mathbb{C}(t)) \,, \qquad n_i \in \mathbb{Q} \,.
\]
 defines an element  $\widetilde{R}^{\QHod}\in \L_{\mathbb{C}\setminus S}^{\QHod}$ for some finite set $S\subseteq \mathbb{C}$. 
 
 We  prove the following two statements by induction on $n$:
 \begin{enumerate}
\item for every $\widetilde{R}\in \mathcal{A}_n(\mathbb{C}(t))$ for $n\geq 2$ with $\delta(R)=0$ in $\bigwedge^2 \mathcal{L}^\mathrm{f}(\mathbb{C}(t))$, $\widetilde{R}^{\QHod}$ is a constant variation,
\item for any $R\in \mathcal{R}_n(\mathbb{C})$ we have $\cor^{\QHod}(R)=0$. 
 \end{enumerate}

For $n=1$ the first statement is void and the second statement is trivial since $\mathcal{R}_1(\mathbb{C})=0$. We assume that $n\geq 2.$ We have 
 \[
 \delta(\widetilde{R})\in \sum_{k=1}^{n-2} \mathcal{A}_k(\mathbb{C}(t))\wedge \mathcal{R}_{n-k}(\mathbb{C}(t))\subseteq \bigwedge\nolimits^2 \mathcal{A}(\mathbb{C}(t)).
 \]
 Next, we use the so-called rigidity argument (see \cite[\S8]{GonPeriods02}, \cite[Lemma 2.7]{GR-zeta4}, and \cite[\S 2.5]{Rud20}), which means that to show that the variation $\widetilde{R}^{\QHod}$ is constant it is sufficient to prove that $\delta \widetilde{R}^{\QHod}$ vanishes in $\bigwedge^2 \L_{\mathbb{C}\setminus S}^{\QHod}$. We have 
\[
\delta(\widetilde{R}^{\QHod})=\sum \widetilde{R}_{(1)}^{\QHod}\wedge \widetilde{R}_{(2)}^{\QHod}
\]
for some
$\widetilde{R}_{(1)}\in \mathcal{A}_k(\mathbb{C}(t))$ and $\widetilde{R}_{(2)}\in  \mathcal{R}_{n-k}(\mathbb{C}(t))$ with  $1\leq k \leq n-2$ (note that $\mathcal{R}_{1}(\mathbb{C}(t))=0$).
By \autoref{LemmaRCoideal}, each element $\widetilde{R}_{(2)}\in  \mathcal{R}_{n-k}(\mathbb{C}(t))\subseteq \mathcal{A}_{n-k}(\mathbb{C}(t))$
satisfies $\delta \widetilde{R}_{(2)}=0$ in $\bigwedge^2\mathcal{L}^\mathrm{f}(\mathbb{C}(t))$. By the induction hypothesis, 
$\widetilde{R}_{(2)}^\QHod$ is a constant variation. Moreover, its specialization at a regular point lies in $\cor^{\QHod}(\mathcal{R}_{n-k}(\mathbb{C}))$, so by inductive assumption in (ii) it is equal to zero. We conclude that  $\delta \widetilde{R}^{\QHod}=0$ in $\L_{\mathbb{C}\setminus S}^{\QHod}$.  This finishes the proof of inductive step for (i). 
 
For $R\in \mathcal{R}_n(\mathbb{C})$ there exists $\widetilde{R}\in \mathcal{A}_n(\mathbb{C}(t))$ with $\delta(\widetilde{R})=0$ in  $\bigwedge^2 \mathcal{L}^\mathrm{f}(\mathbb{C}(t))$ such that 
$R=\Sp_{t\to 0}(\widetilde{R})-\Sp_{t\to 1}(\widetilde{R})$. 
By (i), $\widetilde{R}^{\QHod}$ is a constant variation, so
\[
\Sp_{t\to 0}(\widetilde{R}^{\QHod})=\Sp_{t\to 1}(\widetilde{R}^{\QHod})\,.
\]
Using \eqref{FromulaSpecializationMHCorrelators}, we get that
\[
\cor^{\QHod}(\Sp_{t\to 0}\widetilde{R})=\cor^{\QHod}(\Sp_{t\to 1}\widetilde{R}),
\]
which proves the inductive step for (ii).
\end{proof}

\begin{remark}
In \cite{Gon19}, Goncharov defined the \emph{canonical real period map} 
$p_{\mathbb{R}}\colon \L_n^{\RHod} \lra \mathbb{R}$ and computed the \emph{Hodge correlator} functions 
\[
\cor_{\mathcal{H}}(x_0,\dots,x_n):=p_{\mathbb{R}}(\cor^{\QHod}(x_0,\dots,x_n))\,.
\]
\Autoref{LemmaHodgeRealization} implies that for $n\geq 1$ we have the canonical real period map $p_{\mathbb{R}}\colon \mathcal{L}^\mathrm{f}_{n}(\mathbb{C})\lra \mathbb{R}.$ Assuming that \autoref{ConjecturePolylogsKtheory} holds for $i=1$, we would have a map
\[
\textup{ch}_{n,1}\colon \gr_{\gamma}^n K_{2n-1} (\mathbb{C})\lra \mathcal{L}^\mathrm{f}_n(\mathbb{C}).
\]
Conjecturally, the composition $p_{\mathbb{R}}\circ \textup{ch}_{n,1}$ should be a non-zero rational multiple of the Borel regulator map.
\end{remark}

\subsection{Motivic realization} \label{SectionMotivicRealization}
Assume that $F$ is a number field. Consider the category $\mathrm{MTM}_{F}$ of mixed Tate motives over $F$ (\cite{Lev93}, \cite{Lev98}, \cite{DG05}, see also \cite{Cle24}). This category is equivalent to the category of finite-dimensional graded comodules over a Hopf algebra of framed mixed Tate motives $\mathcal{H}^{\mathcal{M}}(F)$. Let $\L^{\mathcal{M}}(F)$ be the corresponding Lie coalgebra of indecomposables.  For every embedding $\sigma \colon F\to \mathbb{C}$, we have a map ${r_\sigma}\colon\L^{\mathcal{M}}(F)\lra \L^{\QHod}$ called Hodge realization. 
The Hodge realization functor induces a map 
\begin{equation}\label{FormulaBeilinsonRegulator}
(r_{\sigma})_* \colon  \Ext^1_{\textup{MTM}_F}(\Q(0),\Q(n)) \lra \Ext^1_{\mathrm{MHS}_{\Q}}(\Q(0),\Q(n)).
\end{equation}
It is known that $\Ext^1_{\textup{MTM}_F}(\Q(0),\Q(n))$ is isomorphic to $K_{2n-1}(F)_{\Q}$ and $\Ext^1_{\mathrm{MHS}_{\Q}}(\Q(0),\Q(n))$ is isomorphic to $\mathbb{C}/(2\pi i)^n\Q.$ 

We will use the fact that the map 
\begin{equation}\label{FormulaRegulator}
\oplus(r_{\sigma})_* \colon  \Ext^1_{\textup{MTM}_F}(\Q(0),\Q(n)) \lra 
\bigoplus_{\sigma\colon F\to \mathbb{C}}\Ext^1_{\mathrm{MHS}_{\Q}}(\Q(0),\Q(n))
\end{equation}
is injective. This fact comes from the comparison of the regulators of Beilinson and Borel, proven in \cite{Bei84}, see also \cite{Rap88} and \cite{Bur02}. The precise statement is as follows. Let $r_1$ denote the number of real embeddings of $F$, and $r_2$ the number of pairs of complex embeddings. For $n\geq 2$ we define  
\[
d_n=
\begin{cases}
r_1 & \text{if $n$ is odd,} \\
r_1+r_2  & \text{if $n$ is even.}
\end{cases}
\]
We have the following identification induced by the real or imaginary part:
\begin{equation}\label{SpaceOfInvariants}
\bigg(\bigoplus_{\sigma\colon F \to \mathbb{C}}  \mathbb{C}/(2\pi i)^n\R \bigg)^+ \cong \R^{d_n}
\end{equation}
where the symbol $+$ denotes the space of invariants for complex conjugation acting on each $\mathbb{C}/(2\pi i)^n\R$ and on the set of embeddings of $\sigma \colon F\to \mathbb{C}$. The image of the composition
\[
 \Ext^1_{\textup{MTM}_F}(\Q(0),\Q(n)) 
\lra  \bigoplus_{\sigma\colon F\to \mathbb{C}}  \Ext^1_{\mathrm{MHS}_{\Q}}(\Q(0),\Q(n))\cong \bigoplus_{\sigma\colon F \to \mathbb{C}}  \mathbb{C}/(2\pi i)^n\Q \lra \bigoplus_{\sigma\colon F \to \mathbb{C}}  \mathbb{C}/(2\pi i)^n\R 
\]
lies inside \eqref{SpaceOfInvariants}, so we obtain a map
\[
K_{2n-1}(F)_{\Q}\cong  \Ext^1_{\textup{MTM}_F}(\Q(0),\Q(n)) \lra \bigoplus_{\sigma\colon F\to \mathbb{C}}  \Ext^1_{\mathrm{MHS}_{\Q}}(\Q(0),\Q(n)) \lra  \R^{d_n},
\]
 which is known  be a non-zero rational multiple of the Borel regulator, and thus is injective by the result of Borel \cite{Bor77}. This implies 
that the map \eqref{FormulaRegulator} is injective.

In \cite{Gon19} (see also \cite[\S 2.4]{Gon19B}), Goncharov constructed elements in  $\L^{\mathcal{M}}(F)$ called \emph{motivic correlators}
\[
\cor^{\mathcal{M}}(x_0,\dots,x_n)\in \L_n^{\mathcal{M}}(F)
\]
for $x_0,\dots,x_n \in F.$ For every embedding $\sigma\colon F\lra \mathbb{C}$ we have
\[
r_{\sigma}(\cor^{\mathcal{M}}(x_0,\dots,x_n))=\cor^{\QHod}(\sigma(x_0),\dots,\sigma(x_n)).
\]
Similarly to the Hodge case,  we have a morphism of Lie coalgebras
\[
\cor^{\M}\colon \mathcal{A}(F)\longrightarrow \L^{\M}(F)\,
\]
sending $\llp x_0,\dots,x_n \rrp$ to $\cor^{\mathcal{M}}(x_0,\dots,x_n)\in \L_n^{\mathcal{M}}(F).$

\begin{proposition}\label{LemmaMotivicRealization} Let $F$ be a number field. For any $R\in \mathcal{R}_n(F)$ we have $\cor^{\M}(R)=0$. So, we have a well-defined map 
\[
r_{\M}\colon \mathcal{L}^\mathrm{f}_n(F)\longrightarrow \L_n^{\M}(F)
\]
called the motivic realization.
\end{proposition}

\begin{proof} We argue by induction on $n$. For $n=1$ both $\L_1^\mathrm{f}(F)$ and  $\L_1^{\mathcal{M}}(F)$ are known to be isomorphic to $F^{\times}_{\Q}$ and $\cor(x_0,x_1)=\cor^{\mathcal{M}}(x_0,x_1)=(x_0-x_1)\in F^{\times}_\mathbb{Q}$.
For $n\geq 2$, consider an element 
\[
R=\sum_i n_i\llp f_0^i(t),\dots,f_n^i(t)\rrp\in \mathcal{A}_n(F(t)) \,, \qquad n_i \in \mathbb{Q} \,,
\] 
with $\delta(R)\in \sum_{k=1}^{n-2} \mathcal{A}_k(F(t))\wedge\mathcal{R}_{n-k}(F(t))$.  For $a\in F$, we put 
\[
R_a=\cor^{\mathcal{M}}(\Sp_{t\to a}(R))\in \L^{\mathcal{M}}_n(F)
\]
We need to show that $R_0=R_1$.
By \autoref{LemmaSpecializationCoproduct} and \autoref{LemmaSpecializationRelations}, for every $a\in F$ we have 
\[
\delta\Big(\sum n_i\Sp_{t\to a}\cor(f_0^i(t),\dots,f_n^i(t))\Big)\in \sum_{j=1}^{n-1}\mathcal{A}_j(F)\wedge \mathcal{R}_{n-j}(F)\,,
\]
so, by the induction hypothesis, $\delta(R_a)=0$ and thus  $\delta(R_0-R_1)=0$. Thus 
\[
R_0-R_1\in \Ext^1_{\mathrm{MTM}_F}(\Q(0),\Q(n))\,.
\]

Consider an embedding $\sigma\colon F\to \mathbb{C}$.  Then, by \autoref{LemmaHodgeRealization}, the element $r_{\sigma}(R_0-R_1)=0$. It remains to notice that the map \eqref{FormulaRegulator}
is injective. It follows that $R_0=R_1$.
\end{proof}

\bibliographystyle{halpha2}      
\bibliography{bibliography} 
\end{document}